\documentclass[10pt, leqno]{amsart}
\usepackage{amssymb}
\usepackage{amsmath,amscd}
\usepackage[initials,nobysame,alphabetic]{amsrefs}
\usepackage{amsmath, amssymb}
\usepackage{amsfonts}
\usepackage{mathrsfs}
\usepackage{mathpazo}
\usepackage[arrow,matrix,curve,cmtip,ps]{xy}

\usepackage{amsthm}

\usepackage{float}
\usepackage{hyperref}
\usepackage{graphics}
\usepackage{graphicx}
\usepackage{verbatim}
\usepackage{multirow}
\usepackage{tikz}
\usepackage{enumerate}
\usepackage{eufrak}
\allowdisplaybreaks

\newtheorem{theorem}{Theorem}[section]
\newtheorem{lemma}[theorem]{Lemma}
\newtheorem{proposition}[theorem]{Proposition}
\newtheorem{corollary}[theorem]{Corollary}

\newtheorem*{theorem*}{Theorem}
\newtheorem{remark}[theorem]{Remark}
\newtheorem{definition}[theorem]{Definition}
\newtheorem{example}[theorem]{Example}

\def\as{\hbox{\rm a.s.{ }}}

\numberwithin{equation}{section}

\newcommand{\E}{\mathbb{E}}
\newcommand{\R}{\mathbb{R}}
\newcommand{\Ff}{\mathbb{F}}
\newcommand{\U}{\mathcal{U}}

\newcommand{\V}{\mathcal{V}}
\newcommand{\C}{\mathcal{C}}
\newcommand{\D}{\mathcal{D}}
\newcommand{\F}{\mathcal{F}}

\newcommand{\Om}{\Omega}

\newcommand{\Prob}{\mathbb{\Prob}}

\newcommand{\mytilde}{\raise.17ex\hbox{$\scriptstyle\mathtt{\sim}$}}

\def\b{\beta}                                    
\def\g{\gamma}                                   
\def\d{\delta}             
\def\ep{\varepsilon}       
\def\s{\sigma}                    
\def\t{\theta}                    
\def\r{\varrho}                                   
    \def\z{\zeta}                  

\def\l{\lambda}

\def\e{\eta}

\def\wt{\widetilde}    \def\wh{\widehat}

\def\ms{\medskip} 
\def\no{\noindent}

\begin{document}

\title[Mean-field risk sensitive control and zero-sum games]{ Mean-field risk sensitive control and zero-sum games for Markov chains}

\author{Salah Eddine Choutri and Boualem Djehiche}

\address{Department of Mathematics \\ KTH Royal Institute of Technology \\ 100 44, Stockholm \\ Sweden}
\email{choutri@kth.se, boualem@math.kth.se}


\date{ This version \today}



\begin{abstract}
We establish existence of controlled Markov chain of mean-field type
with unbounded jump intensities by means of a fixed point argument
using the Wasserstein distance. Using a Markov chain entropic backward SDE approach, we further
suggest conditions for existence of an optimal control and a saddle-point
for respectively a control problem and a zero-sum differential game associated
with risk sensitive payoff functionals of mean-field type.
\end{abstract}

\maketitle

\tableofcontents
\section{Introduction}

In this paper we study existence of optimal controls and saddle-points of zero-sum games associated with Markov chains of mean-field type (a.k.a.\@ nonlinear Markov chains). These are  pure jump processes with a discrete state space whose jump intensities further depend on the marginal law of the process. The modeling power of the nonlinear Markov chain in biology, chemistry, economics, physics and communication networks etc.\@ is well documented in the literature, see e.g. \cite{Sch, NP, Chen, Kolo, Oel, DZ, FZ, Fe, Leo1, Leo2, DK, DS}, due to the fact that it is the limit of a system of pure jump processes with mean-field interaction, when the system's size tends to infinity. Its marginal law which satisfies a 'nonlinear' Fokker-Planck equation called the {\it McKean-Vlasov equation} represents the law of a typical trajectory in the underlying collection of interacting jump processes. Optimal controls and games based on the nonlinear Markov chain should give an insight into the effect of the design of control and game strategies for large systems of interacting jump processes.

In this paper, we derived conditions for existence of an optimal control and a saddle-point for respectively a control problem and a zero-sum differential game for nonlinear Markov chains associated with  performance functionals of risk sensitive type. These payoff functionals are obtained by exponentiating the stage-additive performance functional before expectation. Jacobson \cite{Jacobson} was first to show that the risk sensitive  payoff functional is a plausible way to capture  risk-averse and risk-seeking behaviors, that cannot be captured by the risk-neutral performance functional.  
 
 Given a control process $u$ from a suitable class $\U$ of admissible controls, with values in some compact metric space $(U,\delta)$, we consider a controlled probability measure $P^u$ under which $x$ is a pure jump process whose jump intensity from state $i$ to state $j$ at time $t$ is of the functional and mean-field type form $\l_{ij}(t,x,P^u\circ x_t^{-1},u_t)$, where by functional we mean its  dependence on the whole path $x$ and by mean-field type its dependence on  $P^u\circ x_t^{-1}$, the marginal probability distribution of $x_t$ under the probability measure $P^u$, provided it is predictable.  The risk sensitive payoff functional  $J(u),\,\, u\in\U$, associated with the controlled nonlinear Markov chain is 
\begin{equation*}
J(u):=E^u\left[\exp{\left(\int_0^T f(t,x,P^u\circ x_t^{-1},u_t)dt+ h(x_T,P^u\circ x_T^{-1})\right)}\right],
\end{equation*}
where $E^u$ denotes the expectation w.r.t. $P^u$. Any admissible control $u^*$ satisfying
\begin{equation}\label{opt-J}
J(u^*)=\min_{u\in\U}J(u)
\end{equation}
is called optimal control. We want to show existence of such an optimal control. We also consider a mean-field  risk-sensitive zero-sum game between two players.

In \cite{choutri1} a solution to this type of control and zero-sum game problems associated with risk-neutral mean-field payoff functionals was derived, where existence and uniqueness of the underlying mean-field chain were established using a fixed-point argument based on the Girsanov transform and the Csisz{\'a}r-Kullback-Pinsker inequality between the total variation (TV) distance and the entropy (Hellinger) distance, which required that the jump intensities are bounded from below by a strict positive constant. Since TV does not guarantee existence of finite moments, mean-field couplings of the type $E^u\left[X_t^u\right]$ or $E^u\left[\varphi(X_t^u)\right]$ where $\varphi$ is a Lipschitz function, were excluded. To consider this type of couplings, the Wasserstein metric turns out more appropriate as it is designed to guarantee finite moments. But, then we can no longer use the approach of \cite{choutri1}, based on the Girsanov transform because, in general, there is no relation between the Wasserstein metric and the Hellinger (Entropy) distance, unless the nonlinear Markov chain satisfies a log-Sobolev inequality, see for instance \cite{BGL} and \cite{OV} for further details. Such a log-Sobolev inequality will not be studied in this paper.

Under mild integrability and growth conditions on the unbounded jump intensities, using the Wasserstein metric,  we show that by applying the Skorohod selection (or embedding) theorem and  $L^2$-estimates, a fixed-point argument is still valid to derive existence and uniqueness of $P^u$. This turns out possible, thanks to Ekeland's distance on the set of admissible controls which makes it complete (or Polish space). Existence of an optimal control and a saddle-point of the game are derived using techniques involving Markov chain entropic backward stochastic differential equations (BSDE) which boils down to finding a minimizer and a min-max of an underlying Hamiltonian $H$. As documented in \cite{choutri1}, since the mean-field coupling through the marginal law of the controlled chain makes the Hamiltonian $H$, evaluated at time $t$, depend on the whole path of the control process over the time interval $[0,t]$, we cannot follow the frequently used procedure in standard optimal control and perform a {\it deterministic} minimization of  $H$ over the set of actions $U$ and then apply a Bene{\v{s}-type progressively measurable selection theorem to produce an optimal control. We should rather take  the essential infimum of $H$ over the set $\U$ of progressively measurable controls. This nonlocal feature of the dependence of $H$ on the control does not seem covered by the existing powerful  measurable selection theorem. Therefore, our main results are  formulated {\it by assuming} existence of an essential minimum $u^*\in \U$ of $H$ and use suitable comparison results of Markov chain BSDEs to show that $u^*$ is in fact an optimal control, simply because we don't  know of any suitable measurable selection theorem that would guarantee  existence of an essential minimizer of $H$. One should solve this problem on a case-by-case basis. \cite{choutri1} discusses conditions for existence of a nearly-optimal control and  examples where an optimal control exists e.g. provided  the set of Girsanov densities, indexed by admissible controls, is weakly sequentially compact. These cases  are still valid for the risk sensitive case, but we do not repeat them here. 

The structure of the paper is as follows. After a section of preliminaries, we introduce in Section 3 the class of Markov chains of mean-field type under study and prove its existence and uniqueness under rather weak conditions on the underlying unbounded jump intensities.  In Section 4, we consider the  control problem and provide conditions for existence of an optimal control. Finally, in Section 5, we consider a related zero-sum game and derive conditions for existence of a saddle-point under the so-called Isaacs' condition.

\section{Preliminaries}
Let $I=\{0,1,2,\ldots\}$ equipped with its discrete topology and $\s$-field and let $\Om:=\D([0,T],I)$ be the space of functions from $[0,T]$ to $I$ that are right continuous with left limits at each $t\in [0,T)$ and are left continuous at time $T$ endowed with the Skorohod metric $d_0$ which makes $(\Om,d_0)$ a complete separable metric (i.e. Polish) space.  Given $t\in [0,T]$ and $\omega\in\Om$, we put  $x(t,\omega)\equiv\omega(t)$ and denote by $\F^0_t:=\sigma(x(s),\,\, s\le t),\, 0\le t\le T,$ the filtration generated by $x$. The Borel $\sigma$-field $\F$  over $\Om$ coincides with $\sigma( x(s),\,\, 0\le s\le T)$. We also set $|x|_t:=\underset{0\le s\le t}\sup|x(s)|,\,\,\, 0\le t\le T$. 

To $x$ we associate the indicator process $I_i(t)=\mathbf{1}_{\{x(t)=i\}}$ whose value is $1$ if the chain is in state $i$ at time $t$ and $0$ otherwise and the counting processes $N_{ij}(t),\,\,i\ne j$, such that 
$$
N_{ij}(t)=\#\{\tau\in(0,t]:x(\tau^-)=i, x(\tau)=j\},\quad N_{ij}(0)=0,
$$
which count the number of jumps from state $i$ into state $j$ during the time interval $(0,t]$. Obviously, since $x$ is right continuous with left limits, it holds that both $I_i$ and $N_{ij}$ are right continuous  with left limits. Moreover, by the relationship
\begin{equation}\label{x-rep-1}
x(t)=\sum_i iI_i(t),\quad I_i(t)=I_i(0)+\underset{j:\, j\neq i}\sum\left(N_{ji}(t)-N_{ij}(t)\right),
\end{equation}
the state process, the indicator processes, and the counting processes carry the same information which is represented by the natural filtration  $\Ff^0:=(\F^0_t,\, 0\le t\le T)$ of $x$.

Note that \eqref{x-rep-1} is equivalent to the following useful representation
\begin{equation}\label{x-rep-2}
x(t)=x(0)+\sum_{i,j: \,i\neq j} (j-i) N_{ij}(t).
\end{equation}

\ms\no Below,  $C$  denotes a  generic positive constant which may change from line to line.

\subsection{Markov chains} Let $G(t)=(g_{ij}(t),\,i,j\in I),\,0\le t\le T,$ be the predictable $Q$-matrix, i.e.
$g$ is an $\Ff$-predictable process, satisfying 
\begin{equation}\label{G}
g_{ii}(t)=-\underset{j:\, j\neq i}\sum g_{ij}(t),\quad g_{ij}(t)\ge c_1> 0, \quad E\left[\underset{i,j: \, j\neq i}\sum \int_{(0,T]}g_{ij}(t)\, dt\right]<+\infty.
\end{equation}
The assumption that $g_{ij}$ is lower bounded away from zero is imposed to eliminate zero off-diagonal entries of $G$.
In view of e.g. Theorem 4.7.3 in \cite{EK} or Theorem 20.6 in \cite{RW} (for the finite state-space and time independent case with deterministic $Q$-matrix), given the $Q$-matrix $G$ and a probability measure $\xi$ over $I$,  there exists a unique probability measure $P$ on $(\Om,\F)$ under which the coordinate process $x$ is a time-inhomogeneous pure jump process (or chain) with intensity matrix $G$ and starting distribution $\xi$  i.e. such that $P\circ x^{-1}(0)=\xi$. Equivalently, $P$ solves the martingale problem for $G$ with initial probability distribution $\xi$ meaning that,   
for every $f$ on $I$, the process defined by
\begin{equation}\label{f-mart-1}
M^f(t):=f(x(t))-f(x(0))-\int_{(0,t]}(G(s)f)(x(s))\,ds
\end{equation}
is a local martingale relative to $(\Om,\F,\Ff^0)$, where 
$$
G(s)f(i):=\sum_j g_{ij}(s)f(j)=\sum_{j: \,j\neq i}g_{ij}(s)(f(j)-f(i)),\,\,\, i\in I,
$$
and
\begin{equation}\label{G-f}
G(s)f(x(s))=\sum_{i,j: \,j\neq i}I_i(s)g_{ij}(s)(f(j)-f(i)).
\end{equation}
If the $Q$-matrix $G$ is deterministic, $x$ become a time-inhomogeneous Markov chain.

By Lemma 21.13 in \cite{RW}, the compensated processes associated with the counting processes $N_{ij}$ defined by
 \begin{equation}\label{mart-1}
M_{ij}(t)=N_{ij}(t)-\int_{(0,t]} I_i(s^-)g_{ij}(s)\, ds,\quad M_{ij}(0)=0,
\end{equation}
are zero mean, square integrable and mutually orthogonal $P$-martingales  whose  predictable quadratic variations are
\begin{equation}\label{mart-2}
\langle M_{ij}\rangle_t=\int_{(0,t]} I_i(s^-)g_{ij}(s)\, ds.
\end{equation} 
Moreover, at jump times $t$ we have
\begin{equation}\label{mart-3}
\Delta M_{ij}(t)=\Delta N_{ij}(t)=I_i(t^-)I_j(t).
\end{equation}
Thus, the optional variation of $M$ 
$$
[M]_t=\sum_{0<s\le t}|\Delta M(s)|^2=\underset{0<s\le t}\sum\,\underset{i,j:\, j\neq i}\sum|\Delta M_{ij}(s)|^2=\underset{0<s\le t}\sum\,\Delta M_{x(s^-),x(s)}(s)
$$
is
\begin{equation}\label{optional-M}
[M]_t=\underset{0<s\le t}\sum\,\underset{i,j:\, j\neq i}\sum I_i(s^-)I_j(s)=\underset{i,j:\, j\neq i}\sum N_{ij}(t).
\end{equation}
Moreover, in view of \eqref{G},
\begin{equation}\label{exp-optional-M}
E\left[[M]_t\right]\le E\left[\underset{i,j:\, j\neq i}\sum \int_{(0,t]}g_{ij}(s)ds\right]<+\infty.
\end{equation}
We call  $M:=\{M_{ij},\,\, i\neq j\}$ the accompanying martingale of the counting process $N:=\{N_{ij},\,\, i\neq j\}$ or of the chain $x$. 

We denote by $\Ff:=(\F_t)_{0\le t\le T}$ the completion of $(\F^0_t)_{t\le T}$ with the $P$-null sets of $\Omega$. For simplicity, in Sections 4 and 5 below, we will eventually assume that $\F_0$ is trivial. Hereafter, a process from $[0,T]\times\Om$ into a measurable space is said predictable (resp. progressively measurable) if it is predictable (resp. progressively measurable) w.r.t. the predictable $\sigma$-field on $[0,T]\times \Om$ (resp. $\Ff$).

\ms\no For a real-valued matrix $m:=(m_{ij},\, i,j \in I)$ indexed by $I\times I$, we let 
\begin{equation}\label{g-t}
\|m\|_g^2(t):=\underset{i,j:\, i\neq j}\sum |m_{ij}|^2g_{ij}(t)\mathbf{1}_{\{w(t^-)=i\}}<\infty.
\end{equation}
If $m$ is time-dependent, we simply write $\|m(t)\|_g^2$.

\ms\no Let $(Z_{ij}, \, i\neq j)$ be a family of predictable processes and set 
\begin{equation}\label{norm}
\|Z(t)\|^2_{g}:=\sum_{i,j:\, i\neq j}Z^2_{ij}(t)g_{ij}(t)I_{i}(t^-),\quad 0< t\le T,
\end{equation}
\begin{equation}\label{quadratic}
\sum_{0<s\le t}Z(s)\Delta M(s):=\sum_{0<s\le t}\, \underset{i,j: \,i\neq j}\sum\,Z_{ij}(s)\Delta M_{ij}(s)=\sum_{0<s\le t}Z_{x(s^-),x(s)}(s). 
\end{equation}
Consider the local martingale
\begin{equation}\label{stoch}
W(t)=\int_{(0,t]}Z(s)dM(s):=\sum_{i,j: \, i\neq j}\int_{(0,t]} Z_{ij}(s)dM_{ij}(s). 
\end{equation}
Then, the optional  quadratic  variation of the local martingale $W$ 
\begin{equation*}
[W]_t=\sum_{0<s\le t}|Z(s)\Delta M(s)|^2=\sum_{0<s\le t}\,\sum_{i,j: \,i\neq j}Z^2_{ij}(s)|\Delta M_{ij}(s)|^2
\end{equation*}
is
\begin{equation}\label{optional}
[W]_t=\sum_{0<s\le t}\,Z^2_{x(s^-),x(s)}(s)\Delta M_{x(s^-),x(s)}(s)
\end{equation}
and its compensator is
\begin{equation}\label{compensator}
\langle W\rangle_t=\int_{(0,t]}\|Z(s)\|^2_{g}ds.
\end{equation}
Provided that
\begin{equation}\label{Z-g-int}
E\left[\int_{(0,T]} \|Z(s)\|^2_{g}ds\right]<\infty,
\end{equation}
$W$ is a square-integrable martingale if and only if $E\left[[W]_T\right]<+\infty$ if and only if $E[\langle W\rangle_T]<+\infty$. In this case, we have 
\begin{equation}\label{isometry}
E[W^2(t)]=E\left[[W]_t\right]=E[\langle W\rangle_t], \quad 0\le t\le T.
\end{equation} 
Moreover, the following Doob's inequality holds:
\begin{equation}\label{Doob}
E\left[\sup_{0\le t\le T}\left|\int_{(0,t]}Z(s)dM(s)\right|^2\right]\le 4 E\left[\int_{(0,T]} \|Z(s)\|^2_{g}ds\right].
\end{equation}
If $\wt Z$ is another predictable process that satisfies \eqref{Z-g-int}, setting
\begin{equation}
\langle Z(t),\wt Z(t)\rangle_{g}:=\underset{i,j: \,i\neq j}\sum\,Z_{ij}(t)\wt Z_{ij}(t)I_{i}(t^-)g_{ij}(t),\quad 0\le t\le T,
\end{equation}
and considering the martingale
$$
\wt W(t)=\int_{(0,t]}\wt Z(s)dM(s):=\sum_{i,j: \, i\neq j}\int_{(0,t]}\wt Z_{ij}(s)dM_{ij}(s),
$$
it is easy to see that
\begin{equation}\label{covariation}
E\left[ [W,\wt W]_t\right]=E\left[\int_{(0,t]} \langle Z(s),\wt Z(s)\rangle_{g}ds\right].
\end{equation} 


\subsection{Markov chain BSDEs}
 Our approach to show existence of an optimal control and a value of the zero-sum game is based on solutions $(Y,Z)$ of Markov chain backward stochastic differential equations (BSDEs) with data $(\phi,\zeta)$ defined on $(\Omega, \mathcal{F},\mathbb{F},P)$ by
\begin{equation}\label{Y-bsde0}
-dY(t)=\phi(t,\omega,Y(t^-),Z(t))dt-Z(t)dM(t),\quad Y(T)=\zeta.
\end{equation}
\begin{definition}\label{bsde-sol} A solution $(Y,Z)$ of the BSDE \eqref{Y-bsde0} consists of an adapted process $Y$ which is right-continuous with left limits and a predictable process $Z$ which satisfy  
\begin{equation*}
E\left[|Y|_T^2+\int_{(0,T]} \|Z(s)\|^2_{g}ds\right]<\infty.
\end{equation*}  
Uniqueness of this solution occurs $P$-$\as$ for $Y$ and equality  $dP\times g_{ij}(s^-)I_i(s^-)ds$-almost everywhere for $Z$. 
\end{definition}
If $(Y,Z)$ solves \eqref{Y-bsde0} then by taking conditional expectation w.r.t. $\F_t$, we obtain the following representation:
\begin{equation*}
Y(t)=E\Big[\zeta+\int_{(t,T]} \phi(s,\omega,Y(s^-), Z(s))\,ds\Big|\mathcal{F}_t\Big], \quad t\in[0,T].
\end{equation*}
Moreover, $t\mapsto Y(t)$ is right-continuous with left limits. Therefore,  $Y(t^-)=Y(t)\, dt$-a.e. Hence, we may write
\begin{equation*}
Y(t)=E\Big[\zeta+\int_{(t,T]} \phi(s,\omega,Y(s), Z(s))\,ds\Big|\mathcal{F}_t\Big], \quad t\in[0,T].
\end{equation*}
For existence and uniqueness results of solutions of Markov chain BDSEs \eqref{Y-bsde0} based on the martingale representation theorem ($L^2$-theory) we refer to the  series of papers by Cohen and Elliott (see e.g. \cite{Cohen2012} and \cite{Cohen2015} and the references therein). 

Below, we establish existence of an optimal control and a saddle-point for the zero-sum game using some properties of a class of linear  BSDEs for which $\zeta$ is a bounded random variable and the driver $\phi$ is of the form 
\begin{equation}\label{driver}
\phi(t,x,y,z):=\kappa (t,x)y+\g(t,x,z)\quad \text{with}\,\,\,\g(t,x,0)=0,
\end{equation}
where
\begin{enumerate}
\item[(A1)] $\kappa$ is a bounded  and progressively measurable process,

\item[(A2)] $P$-a.s., for all $t\in[0,T],~ y\in\R,~ z^1=(z^1_{ij}),z^2 =(z^2_{ij}),\,\, z^1_{ij}, z^2_{ij}\in \R$
\begin{equation}\label{balancing}
\phi(t,x,y,z^1)-\phi(t,x,y,z^2)=\langle \ell(t,x,z^1,z^2),z^1-z^2\rangle_g,
\end{equation}
 for some predictable process $\ell=(\ell_{ij},\, i,j \in I)$  such that for every  $t\in[0,T]$,
$$
  \|\ell(t)\|_g\le a(t),\quad P\text{-}\as,
$$  where $(a(t))_t$ is a non-negative predictable process which belongs to $L^2([0,T]\times\Omega,dt\otimes dP)$. Moreover,
there exists a probability measure $\wt P$ on $(\Om,\F)$ under which the  processes 
\begin{equation}\label{m-m-bsde}
\wt M_{ij}(t)=M_{ij}(t)-\int_{(0,t]} \ell_{ij}(s)I_i(s^-)g_{ij}ds
\end{equation}
are zero mean, square integrable and mutually orthogonal $\wt P$-martingales.
\end{enumerate}

\ms
The relation \eqref{balancing} is called  condition $(A_{\gamma})$ in \cite{royer} and in \cite{Cohen2015} $\phi$ is called 'balanced'. It constitutes the key assumption which makes the following comparison result for solutions of Markov chain BSDEs possible. For a proof see \cite{royer} and \cite{Cohen2015}. 

\begin{proposition}[{\bf Comparison theorem}]\label{comparison} Let $(\zeta, \phi)$ and $(\wt\zeta,\wt\phi)$ be input data for two BSDEs of the form \eqref{Y-bsde0}, with solutions $(Y,Z)$ and $(\wt Y,\wt Z)$ respectively. Suppose
\begin{itemize}
\item [(a)] $\zeta \ge \wt\zeta \quad P\text{-}\as$,
\item [(b)] $\phi(t, x,y, z) \ge \wt\phi(t, x,y, z)\quad dt \otimes dP\text{-}\as$ for all $(y, z)$\\ 
and at least one of  $\phi$ and  $\wt\phi$ satisfies \eqref{balancing}. 
\end{itemize}
Then 
$$
Y\ge \wt Y \quad P\text{-}\as
$$
\end{proposition}

Let $(Y,Z)$ be a solution of the BSDE \eqref{Y-bsde0} with driver \eqref{driver}. Since $\g(t,x,0)=0$, by \eqref{balancing}, we may write
\begin{equation}\label{driver-linear}
\phi(t,x,y,z)=\kappa (t,x)y+\langle \ell(t,x,z,0),z\rangle_g,
\end{equation}
where $\ell(t,x,Z(t),0)$ is such that \eqref{m-m-bsde} holds. Then $Y$ admits the explicit representation
\begin{equation}\label{Y-rep}
Y(t)=\wt{E}\left[\zeta e^{\int_{(t,T]} k(s,x)ds}|\F_t\right],
\end{equation}
where the conditional expectation is taken w.r.t. $\wt P$.

In the next proposition we summarize existence and uniqueness of solutions of the BSDE \eqref{Y-bsde0} with  driver of the form \eqref{driver}. 

\begin{proposition}\label{bsde-th}
Let  $\phi$ be of the form \eqref{driver} and satisfies the assumptions (A1) and (A2). Moreover, let  $\zeta$ be an $\F_T$-measurable and bounded random variable. Then, the BSDE \eqref{Y-bsde0} associated with $(\phi,\zeta)$ admits a unique solution $(Y,Z)$ for which $Y$ satisfies \eqref{Y-rep} (thus  bounded) and 
\begin{equation*}
E\left[\int_{(0,T]} \|Z(s)\|^2_{g}ds\right]<\infty.
\end{equation*}  
\end{proposition}

The proof of the theorem is similar to that of the Brownian motion driven BSDEs derived in \cite{Ham-Lepl95}, Theorem I-3, using an approximation scheme by  monotone sequences of solutions of standard Markov chain BSDEs for which existence, uniqueness and comparison results (see Proposition \eqref{comparison}) are similar to that of the Brownian motion driven BSDEs derived in \cite{pardoux} and \cite{EPQ}, along with the properties \eqref{isometry} and \eqref{Doob} related to the martingale $W$ displayed in \eqref{stoch} together with It\^o's formula for semimartingales driven by jump processes. We omit the details.

\subsection{The Wasserstein distance for Probability measures on $I$}

\medskip
Let $\mathcal{P}(I)$ denote the set of probability measures on $I$. 

For $\mu,\nu\in\mathcal{P}(I)$, the 2-Wasserstein distance is defined by the formula
\begin{equation}\label{W-2}
d(\mu,\nu):=\inf\left\{\left(\int_{I\times I}|x-y|^2 F(dx,dy)\right)^{1/2}\right\}
\end{equation}
over $F\in\mathcal{P}(I\times I)$ with marginals $\mu$ and $\nu$. It has also the following formulation in terms of a coupling between two random variables $X$ and $Y$ defined on the same probability space:
\begin{equation}\label{d-coupling}
d(\mu,\nu)=\inf\left\{\left(\E\left[|X-Y|^2\right]\right)^{1/2},\,\,\text{law}(X)=\mu,\,\text{law}(Y)=\nu \right\}.
\end{equation}

\medskip\noindent The 1-Wasserstein (or Kantorovich-Rubinstein) distance is defined by the formula
\begin{equation}\label{K-R}
d_1(\mu,\nu):=\inf\left\{\int_{I\times I }|x-y|F(dx,dy)\right\}
\end{equation}
over $F\in\mathcal{P}(I\times I)$ with marginals $\mu$ and $\nu$. It has the following dual representation
\begin{equation}\label{K-R-duality}
d_1(\mu,\nu)=\sup_{\|\psi\|_{Lip}\le 1}\left\{\int_{I} \psi\,d\mu-\int_{I} \psi\,d\nu\right\}
\end{equation}
over Lipschitz functions $\psi$ with  Lipschitz constant less or equal to one. This distance is very natural when the jump intensities are e.g. of the type $\lambda_{ij}(t,x,\int y\mu(dy))$. By the  Cauchy-Schwarz inequality we have 
$$
d_1(\mu,\nu)\le d(\mu,\nu).
$$

\medskip\noindent
Similarly, on $(\Om,\Ff)$ we define the 2-Wasserstein  metric between two probability measures $P$ and $Q$ as
\begin{equation}\label{W-filt}
D_t(P,Q):=\inf\left\{\left(\int_{\Om\times\Om}|x-y|_t^2 R(dx,dy)\right)^{1/2}\right\},\quad 0\le t\le T,
\end{equation}
over $R\in\mathcal{P}(\Om\times \Om)$ with marginals $P$ and $Q$.

\smallskip\noindent  We have
\begin{equation}\label{ordering}
D_s(P,Q)\le D_t(P,Q),\quad 0\le s\le t.
\end{equation}
Moreover,  for $P, Q\in \mathcal{P}(\Om)$ with time marginals $P_t:=P\circ x^{-1}(t)$ and $Q_t:=Q\circ x^{-1}(t)$, the  2-Wasserstein  distance between $P_t$ and $Q_t$ satisfies
\begin{equation}\label{margine}
d(P_t,Q_t)\le D_t(P,Q),\quad 0\le t\le T.
\end{equation}

Endowed with the 2-Wasserstein metric $D_T$, $\mathcal{P}_2(\Om)$ is a complete metric space. Moreover, $D_T$ carries out the usual topology of weak convergence.


\section{Existence of controlled mean-field Jump processes}
In this section we  show existence and uniqueness of controlled jump processes of mean-field type using the 2-Wasserstein distance as a carrier of the topology of weak convergence. A construction of such processes using the total variation distance is given in \cite{choutri1}. 

\ms Let $(U, \delta)$ be a compact metric space with its Borel field $\mathcal{B}(U)$ and $\U$ the set of $\Ff$-progressively measurable processes $u=(u(t),\,0\le t\le T)$ with values in $U$. We call $\U$ the set of admissible controls.

\ms\no
We would like to show that, for  each $u\in\U$, there exists a unique probability measure $P^u$ on $(\Om,\F)$ under which the coordinate process $x$ is a jump process with intensities 
\begin{equation}\label{u-lambda}
\l^u_{ij}(t):=\l_{ij}(t,x,P^u\circ x^{-1}(t),u(t)),\,\,\, i,j\in I,\,\, 0\le t\le T.
\end{equation}

We assume the following. 
\begin{itemize}
 \item[(B1)] For any $Q\in \mathcal{P}(\Omega)$, $u\in\U$ and  $i,j\in I$,  the process $((\l_{ij}(t, x,Q\circ x^{-1}(t),u(t)))_t$ is predictable and satisfies, for every $t\in[0,T],\,w\in \Omega,\,\mu\in\mathcal{P}(I)$ and $i\neq j$,
 $$
 \underset{u\in U}{\inf\,}\l_{ij}(t, w,\mu,u)>0.
 $$
 
 \item[(B2)] For $p=1,2$ and for every $t\in[0,T]$, $w\in \Om,\, u\in U$ and  $\mu\in \mathcal{P}_2(I)$, 
 $$
 \underset{i,j: \, j\neq i}\sum |j-i|^p\l_{ij}(t,w,\mu,u)\le C(1+|w|^p_t+\int|y|^p\mu(dy)).
 $$
 
 \item[(B3)]  For $p=1,2$ and for every $t\in[0,T]$, $w, \tilde w\in \Om$ and  $\mu, \nu \in\mathcal{P}(I)$,
  $$
  \underset{i,j: \, j\neq i}\sum |j-i|^p|\l_{ij}(t,w,\mu,u)-\l_{ij}(t,\tilde w,\nu,v)| \le C(|w-\tilde w|^p_t+d^p(\mu,\nu)+\delta^p(u,v)).
  $$
  
  \item [(B4)] The following holds for the $Q$-matrix $(g_{ij})$ and the probability measure $\xi$ on $I$. 
 $$
 \underset{i,j: \, j\neq i}\sum |j-i|^2g_{ij}<\infty,\quad \|\xi\|_2^2:=\int |y|^2\xi(dy)<\infty.
 $$
 \end{itemize}

\begin{remark}\label{Lip-u} 
 \begin{enumerate}
\item The positivity of $\l$ in (B1) is imposed to enable the construction of $P^u$ by  using a Girsanov-type change of probability measure.
\item 
 The Lipschitz continuity condition (B3) of $\l$ w.r.t.\@ the control parameter is used to only show that the map $u\mapsto P^u$ is continuous. It will not be needed to prove the main result, Theorem \eqref{opt-BSDE}. 
 \end{enumerate}
\end{remark}

Examples of intensities satisfying (B2) and (B3) include the following class of mean-field versions of the so-called reaction models of polynomial type (see e.g. \cite{Chen}, pp. 460-463).
 
\begin{example} {\bf Mean-field Schl\"ogl and Autocatalytic models}. In the mean-field  version of the Schl\"ogl model (cf. \cite{NP}, \cite{Chen}, \cite{DZ}, \cite{FZ} and \cite{Fe}) the intensities are of the form: 
\begin{itemize}
\item {\bf Schl\"ogl's first model}
 \begin{equation*}\label{schlogl1}
 \l_{ij}(u,w,\mu):=\left\{\begin{array}{ll} \nu_{ij} & \text{if} \,\, j\neq i+1, i-1,\\
 \b_0+\b_1\int y\mu(dy)  & \text{if} \,\, j= i+1,\\
 \d_1 \int y\mu(dy)+\d_2 \int y(y-1)\mu(dy) & \text{if} \,\, j= i-1,
 \end{array}
 \right.
 \end{equation*}
 where the control parameter is $u:=(\b_0,\b_1,\d_1,\d_2)$. When $\b_0=\d_1=0$ we obtain the Autocatalytic  model. \\
 
\item  {\bf Schl\"ogl's second  model} 
 \begin{equation*}\label{schlogl2}
 \l_{ij}(u,w,\mu):=\left\{\begin{array}{ll} \nu_{ij} & \text{if} \,\, j\neq i+1, i-1,\\
 \b_0+\b_2\int y(y-1)\mu(dy)  & \text{if} \,\, j= i+1,\\
 \d_1 \int y\mu(dy)+\d_3 \int y(y-1)(y-2)\mu(dy) & \text{if} \,\, j= i-1.
 \end{array}
 \right.
 \end{equation*}
 The control parameter is $u:=(\b_0,\b_2,\d_1,\d_3)$. This model requires the use of the 3-Wasserstein metric. \\
 \end{itemize}
 In these examples, the entries of the control parameter are all strictly positive.  Moreover, $(\nu_{ij})_{ij}$ is a deterministic $Q$-matrix for which there exists $N_0\ge 1$ such
that $\nu_{ij}= 0$ for $|j-i|\ge N_0$ and $\nu_{ij}> 0$ for $|j-i|< N_0$. 
\end{example} 

\medskip For $u\in\U$, let $Q^u\in \mathcal{P}_2(\Om)$ and $P^u$ be the probability measure on $(\Om,\F)$ for which the coordinate process $x$ is a jump process with intensity  
\begin{equation}\label{Q-u-lambda}
\l_{ij}(t,x,Q^u\circ x^{-1}(t),u(t)),\,\,\, i,j\in I,\,\, 0\le t\le T.
\end{equation}
Such a probability measure exists because the intensity matrix is standard. \\
\no Existence and uniqueness of a mean-field jump process with intensity \eqref{u-lambda} boils down to showing that $P^u=Q^u$ i.e. $Q^u$ is a fixed point. \\  

We have
\begin{theorem}\label{FP-u} Assume $\lambda$ satisfies the conditions (B1)-(B3). Then, for each admissible control process $u\in\U$, the map 
\begin{equation*}\begin{array}{lll}
\Phi: \mathcal{P}_2(\C)\longrightarrow \mathcal{P}_2(\C) \\ 
\qquad\quad Q^u  \longrightarrow \Phi(Q^u):=P^u
\end{array}
\end{equation*} 
under which the coordinate process $x$ is a jump process with intensity \eqref{Q-u-lambda} and initial distribution $\xi$ having finite second moment, is well defined. Moreover, it admits a unique fixed point.
\end{theorem}

\begin{proof} 
The proof uses the Skorohod's representation theorem (see \cite{EK}, Theorem 3.1.8). To this end, the admissible control processes $u\in \U$ should be seen as  random variables taking values in a Polish space. This is possible only if we are able to put a suitable topology on the set of controls. Indeed, denote 
\begin{equation} \label{control-set}
\widehat{\U}:=\{u:[0,T]\longrightarrow U; \,u_t \,\text{measurable} \}.
\end{equation}
By \cite{Ekeland}, Lemma 7.2,  the following metric
\begin{equation}\label{ekeland}
\delta_E(u,v):=\text{meas}(\{t\in[0,T],\,\,\delta(u(t),v(t))>0\})
\end{equation}
is a distance, where '$\text{meas}(A)$' of a subset $A$ of $[0,T]$ denotes its Lebesgue measure. Moreover, the metric space $(\widehat{\U},\delta_E)$ is a separable complete metric (i.e. a Polish) space.  The space $\Omega\times \widehat{\U}$, being Polish, we can apply Skorohod's representation theorem to the pair $(x,u): \Omega\longrightarrow \Omega\times \widehat{\U}$, using the same argument as \cite{EK}, Theorem 6.4.1, as follows. There exists a probability space $(\widehat{\Omega},\widehat{\F},\widehat{P})$ on which are defined a sequence  $N^0_{ij}, \, j\neq i$, of independent Poisson processes with intensity 1, a $\widehat\U$-valued process $\bar u$ and a random variable $\zeta$  such that
\begin{equation}\label{x-Q-u-rep}
x^{Q^{\bar u}}(t)=\zeta+\sum_{i,j: \,i\neq j} (j-i)N^0_{ij}\left(\int_{(0,t]} \l_{ij}(s,x^{Q^{\bar u}},Q^{\bar u}(s),\bar{u}(s) )ds\right)
\end{equation}
and
\begin{equation}
x^{\widetilde{Q}^{\bar u}}(t)=\zeta+\sum_{i,j: \,i\neq j} (j-i)N^0_{ij}\left(\int_{(0,t]}  \l_{ij}(s,x^{\widetilde{Q}^{\bar u}},\widetilde{Q}^{\bar u}(s),\bar{u}(s))ds\right),
\end{equation}
and for which $\bar u$  has the same distribution as $u$,  $x^{Q^{\bar u}}$ (resp. $x^{\widetilde{Q}^{\bar u}}$) has the same distribution as the coordinate process $x$ under $\Phi(Q^u)$ (resp. $\Phi(\widetilde{Q}^u)$), $Q^{\bar u}(t)$ (resp. $\widetilde{Q}^{\bar u}(t)$) is the $t$-marginal distribution of $x^{Q^{\bar u}}$ (resp. $x^{\widetilde{Q}^{\bar u}}$) and $\zeta$ has the same distribution $\xi(dy)$ as $x(0)$ under $\Phi(Q^u)$ and $\Phi(\widetilde{Q}^{u})$. In particular,
\begin{equation}
\|\Phi(Q^u)\|_2^2=\widehat{E}[|x^{Q^{\bar u}}|_T^2].
\end{equation}
Moreover, we have the 'coupling' inequality (cf. \eqref{d-coupling}):
\begin{equation}\label{coupling}
D^2_t(\Phi(Q^u),\Phi(\widetilde{Q}^u))\le \widehat{E}\left[|x^{Q^{\bar u}}-x^{\widetilde{Q}^{\bar u}}|_t^2\right],\quad 0\le t\le T,
\end{equation}
where $\widehat{E}$ denotes the expectation w.r.t. the probability measure $\widehat{P}$ on the new probability space.

Given $Q^u\in \mathcal{P}_2(\Om)$ we first show that $P^u:=\Phi(Q^u) \in \mathcal{P}_2(\Om)$ i.e. $\|P^u\|_2^2=\widehat{E}[|x^{Q^{\bar u}}|_T^2]<\infty$. Since the $N^0_{ij},\,i\neq j$, are mutually independent, we obtain from \eqref{x-Q-u-rep} that
$$
\widehat{E}[|x^{Q^{\bar u}}|_T^2]\le C\left(\|\xi\|_2^2+\sum_{i,j: \,i\neq j} (j-i)^2 \widehat{E}\left[\left(N^0_{ij}\left(\int_{(0,T]}  \l_{ij}(s,x^{Q^{\bar u}},Q^{\bar u}(s),\bar{u}(s))ds\right)\right)^2\right] \right).
$$
But, by the Meyer-Doob decomposition of the time changed processes \\ $N^0_{ij}\left(\int_0^T \l_{ij}(s,x^{Q^{\bar u}}, Q^{\bar u}(s),\bar{u}(s))ds\right)$, we have
$$
\begin{array}{lll}
\underset{i,j: \,i\neq j}{\sum}(j-i)^2 \widehat{E}\left[\left(N^0_{ij}\left(\int_{(0,T]} \l_{ij}(s,x^{Q^{\bar u}},Q^{\bar u}(s),\bar{u}(s))ds\right)\right)^2\right] \\ \,\,\le  \int_{(0,T]}\widehat{E}\left[\underset{i,j: \,i\neq j}{\sum} (j-i)^2 \l_{ij}(s,x^{Q^{\bar u}},Q^{\bar u}(s),\bar{u}(s))  + \left(\underset{i,j: \,i\neq j}{\sum} |j-i| \l_{ij}(s,x^{Q^{\bar u}},Q^{\bar u}(s),\bar{u}(s))\right)^2\right]ds.
\end{array}
$$
Thus, by (B2) we get
\begin{equation}\label{x-u.square}
\widehat{E}[|x^{Q^{\bar u}}|_T^2]\le C\left(1+\|\xi\|_2^2+\|Q^u\|^2_T+\int_{(0,T]} \widehat{E}[|x^{Q^{\bar u}}|_s^2]ds\right)
\end{equation}
and by Grownwall's inequality, we finally have
\begin{equation}\label{x-square}
\|P^u\|_2^2=\widehat{E}[|x^{Q^{\bar u}}|_T^2]\le  Ce^{CT}\left(1+\|\xi\|_2^2+\|Q^u\|^2_T\right),
\end{equation}
which shows that the mapping $\Phi$ is well defined. 

For a positive integer $N$, let $\Phi^N$ denote the $N$-fold composition of the map $\Phi$. If we show that, for $N$ large enough, $\Phi^N$ is a contraction i.e. given $Q:=Q^u$ and $\widetilde{Q}:=\widetilde{Q}^u$ in $\mathcal{P}_2(\C)$,
$$
D^2_T(\Phi^N(Q),\Phi^N(\widetilde{Q}))\le k_N D^2_T(Q,\widetilde{Q}),
$$
for some constant $k_N<1$, then $\Phi$ admits a unique fixed point.\\
\no Indeed, again, since the $N^0_{ij},\,i\neq j$, are mutually independent, we have
$$
\widehat{E}\left[ |x^{Q^{\bar u}}-x^{\widetilde{Q}^{\bar u}}|_t^2\right]=\sum_{i,j: \,i\neq j} (j-i)^2 \widehat{E}\left[\left(N_{ij}(t)-\widetilde{N}_{ij}(t)\right)^2\right],
$$
where, by the stationarity of the Poisson process, the processes 
$$\begin{array}{lcr}
N_{ij}(t):=N^0_{ij}\left(\int_{(0,t]} \l_{ij}(s,x^{{Q}^{\bar u}},Q^{\bar u}(s),\bar{u}(s))ds\right), \\  \widetilde{N}_{ij}(t):=N^0_{ij}\left(\int_{(0,t]} \l_{ij}(s,x^{\widetilde{Q}^{\bar u}},\widetilde{Q}^{\bar u}(s),\bar{u}(s))ds\right)
\end{array}
$$
satisfy
$$
\begin{array}{lll}
\widehat{E}\left[\left(N_{ij}(t)-\widetilde{N}_{ij}(t)\right)^2\right] \\ \,\,=\widehat{E}\left[\left(N^0_{ij}\left(\Big|\int_{(0,t]} \left(\l_{ij}(s,x^{{Q}^{\bar u}},{Q}^{\bar u}(s),\bar{u}(s))-\l_{ij}(s,x^{\widetilde{Q}^{\bar u}},\widetilde{Q}^{\bar u}_s,\bar{u}(s))\right)ds\Big|\right)\right)^2\right]\\ \quad \le \widehat{E}\left[\int_{(0,t]} \Big|\l_{ij}(s,x^{{Q}^{\bar u}},{Q}^{\bar u}(s),\bar{u}(s))-\l_{ij}(s,x^{\widetilde{Q}^{\bar u}},\widetilde{Q}^{\bar u}(s),\bar{u}(s))\Big|ds\right]\\ \quad  +\widehat{E}\left[\int_{(0,t]}\Big|\l_{ij}(s,x^{{Q}^{\bar u}},{Q}^{\bar u}(s),\bar{u}(s))-\l_{ij}(s,x^{\widetilde{Q}^{\bar u}},\widetilde{Q}^{\bar u}(s),\bar{u}(s))\Big|^2\, ds\right].
\end{array}
$$
Therefore, 
$$\begin{array}{lll}
\widehat{E}\left[ |x^{Q^{\bar u}}-x^{\widetilde{Q}^{{\bar u}}}|_t^2\right]\\ \qquad \le \int_{(0,t]} \widehat{E}\left[\underset{i,j: \,i\neq j}{\sum} (j-i)^2 \Big|\l_{ij}(s,x^{{Q}^{\bar u}},{Q}^{\bar u}(s),\bar{u}(s))-\l_{ij}(s,x^{\widetilde{Q}^{\bar u}},\widetilde{Q}^{\bar u}(s),\bar{u}(s))\Big| \right. \\ \quad\quad\qquad \left. + \left(\underset{i,j: \,i\neq j}{\sum} |j-i| \Big|\l_{ij}(s,x^{{Q}^{\bar u}},{Q}^{\bar u}(s),\bar{u}(s))-\l_{ij}(s,x^{\widetilde{Q}^{\bar u}},\widetilde{Q}^{\bar u}(s),\bar{u}(s))\Big| \right)^2\right]ds.
\end{array}
$$
Using (B4), we obtain
\begin{equation}\label{x-estimate}
\widehat{E}\left[ |x^{Q^{\bar u}}-x^{\widetilde{Q}^{\bar u}}|_t^2\right]\le C \int_{(0,t]}\left(\widehat{E}\left[|x^{Q^{\bar u}}-x^{\widetilde{Q}^{\bar u}}|^2_s\right]+d^2(Q^u(s),\widetilde{Q}^u(s))\right)ds.
\end{equation}
Gronwall's inequality yields 
$$
\widehat{E}\left[ |x^{Q^{\bar u}}-x^{\widetilde{Q}^{\bar u}}|_t^2\right]\le C \int_{(0,t]} d^2(Q^u(s),\widetilde{Q}^u(s))ds.
$$
Hence, in view of \eqref{margine} and \eqref{coupling} we have
$$
D^2_t(\Phi(Q),\Phi(\widetilde{Q}))\le C \int_{(0,t]} D_s^2(Q,\widetilde{Q})ds.
$$
Iterating this inequality, we obtain, for every $N>0$,
$$\begin{array}{lll}
D^2_T(\Phi^N(Q),\Phi^N(\widetilde{Q}))\le C^N\int_{(0,T]}\frac{(T-t)^{N-1}}{(N-1)!}D^2_t(Q,\widetilde{Q})dt \le  \frac{C^NT^N}{N!}D^2_T(Q,\widetilde{Q}),
\end{array}
$$
where $\Phi^N$ denotes the $N$-fold composition of the map $\Phi$. Hence, for $N$ large enough, $\Phi^N$ is a contraction which entails that $\Phi$ admits a unique fixed point.  
\end{proof}

In view of (B4), mimicking the proof of \eqref{x-square} we obtain the following 
\begin{corollary} There exists a positive constant $C$, independent of the control $u$, such that
\begin{equation}\label{x-square-g}
\underset{u\in\U}{\sup}\|P^u\|_2^2=\underset{u\in\U}{\sup}E^u[|x|_T^2]\le  Ce^{CT}\left(1+\|\xi\|_2^2\right).
\end{equation}
In particular,
\begin{equation}\label{x-square-g-0}
\|P\|_2^2=E[|x|_T^2]\le  Ce^{CT}\left(1+\|\xi\|_2^2\right).
\end{equation}
\end{corollary}
Again  mimicking the proof of \eqref{x-square} and using the Lipschitz continuity (B3) of the intensity process w.r.t.\@ $(w, u,\mu)$, we further have the following estimate of the Wasserstein distance between $P^u$ and $P^v$. The estimate below uses the expectation w.r.t. the probability measure $\wh{P}$ and $(\bar u,\bar v)$ constructed as above using the Skorohod's representation theorem.   
\begin{lemma}\label{W-u}
For every $u,v\in \U$, it holds that
\begin{equation}\label{W-uv-1}
D_T^2(P^u,P^v)\le Ce^{CT}\wh{E}[\int_{(0,T]}\delta^2(\bar u(t),\bar v(t))dt],
\end{equation}
In particular, the function $u\mapsto P^u$ from $(U,\delta)$ into $(\mathcal{P}_2(\Om), D_T)$ is Lipschitz continuous: for every $u,v\in U$, 
\begin{equation}\label{W-uv-2}
D_T(P^u,P^v)\le Ce^{CT}\delta(u,v).
\end{equation} 
\end{lemma}


\medskip
In the rest of the paper, we let $P$ be the probability under which $x$ is a time-homogeneous Markov chain such that $P\circ x^{-1}(0)=\xi$ and with deterministic time-independent $Q$-matrix $(g_{ij})_{ij}$ satisfying \eqref{G}. 
\section{The risk sensitive control problem}
Given an admissible control $u\in\U$, we consider the probability measure $P^u$ on $(\Omega,\F)$ under which the coordinate process $x$ is a pure jump process with intensity $\l^u(t)=(\l_{ij}^u(t))$, where 
\begin{equation}\label{lambda-u}
\l^u(t):=\l(t,x,P^u\circ x^{-1}(t),u(t)),\quad 0\le t\le T.
\end{equation}
The payoff functional $J(u),\,\, u\in\U$, associated with the controlled probability measure $P^u$ is 
\begin{equation}\label{J-u}
J(u):=E^u\left[\exp{\left(\int_0^T f(t,x,P^u\circ x_t^{-1},u_t)dt+ h(x_T,P^u\circ x_T^{-1}\right)}\right],
\end{equation}

We want to find an optimal control $u^*\in\U$ for which
\begin{equation}\label{opt-J}
J(u^*)=\underset{u\in\U}{\min}J(u),
\end{equation}
and characterize the optimal cost functional $J(u^*)$. The corresponding optimal dynamics is given by the probability measure $P^*$ on $(\Om,\F)$ 
under which the coordinate process $x$ is a pure jump process with intensity $\l^{u^*}(t),\,0\le t\le T$. This is achieved by characterizing the risk sensitive payoff $J(u)$ of the control problem given by \eqref{opt-J} in terms of an entropic BSDE and the comparison result for solutions of Markov chain BSDEs (see Proposition \eqref{comparison}). 

We first recall the definition of an entropic Markov chain BSDE and given some further estimates and properties of the controlled intensities $\l^u$ needed below.

\begin{definition}\label{entropic-bsde}
A process $(Y,Z)$ defined on $(\Om,\mathcal{F},\mathbb{F},P)$ is called solution of an entropic Markov chain BSDE with data $(\phi,\zeta)$ if it satisfies
\begin{equation}\label{Y-bsde1}\left\{\begin{array}{lll}
-dY(t)=\{ \phi(t,\omega,Y(t^-),Z(t))+\langle \tau(Z(t)),1\rangle_g\}dt-Z(t)dM(t),\\  Y(T)=\zeta,
\end{array}
\right.
\end{equation}
where $\tau(z):=e^z-z-1$ is the convex conjugate of the function $\tau^*(z):=z\ln{z}-z+1$ which is the entropy associated with the Poisson process with intensity 1.
\end{definition}

Combining \eqref{G}, (B2), \eqref{x-square-g} and \eqref{x-square-g-0}, we obtain the following
\begin{lemma} There exists a positive constant $C_T$ independent of the controls $u$  such that
\begin{equation}\label{lambda-norm}
E\left[\int_{(0,T]}\mathrm{ess}\sup_{u\in \U}\|\l^u(t)\|^2_gdt\right]\le C_T.
\end{equation}
\end{lemma}
\begin{proof}
We have
$$
\begin{array}{lll}
E[\int_{(0,T]} \underset{u\in \U}{\mathrm{ess}\sup\,} \|\l^u(t)\|^2_gdt]\le \left(\underset{i,j: i\neq j}\sum g_{ij}\right) E\left[\int_{(0,T]} \underset{u\in \U}{\mathrm{ess}\sup\,}\left(\underset{i,j: i\neq j}\sum\l_{ij}^u(t)\right)^2dt\right] \\ \qquad\qquad\qquad\qquad\le C  \left(\underset{i,j: i\neq j}\sum g_{ij}\right)(1+E[|x|^2_{T}]+\sup_{u\in \U}\|P^u\|^2)\\ \qquad\qquad\qquad\qquad \le  C  \left(\underset{i,j: i\neq j}\sum g_{ij}\right)\left(1+E[|x|^2_{T}]+e^{CT}\left(1+\|\xi\|_2^2\right)\right):=C_T.
\end{array}
$$
\end{proof}



\ms\noindent For $(t,u)\in [0,T]\times \U$, set 
$$
\ell^u(t):=\ell(t,x,P^u\circ x^{-1}(t),u(t)),
$$
where
\begin{equation}\label{ell-u}
\ell^u_{ij}(t):=\left\{\begin{array}{rl} \frac{\l^u_{ij}(t)}{g_{ij}} -1 &\text{if }\,\, i\neq j,\\ 0 & \text{if }\,\, i=j,
\end{array}
\right.
\end{equation}
Using \eqref{lambda-norm}, we readily obtain 
\begin{equation}\label{ell-norm}
E\left[\int_{(0,T]}\mathrm{ess}\sup_{u\in \U}\|\ell^u(t)\|^2_gdt\right]\le C_T,
\end{equation}
where $C_T$ is a positive constant $C_T$ independent of the controls $u$.
 
Consider the Dol{\'e}ans-Dade exponential
\begin{equation}\label{exp-mg}
L^u_t:=\underset{\substack{i,j\\ i\neq j}}\prod \exp{\left\{ \int_{(0,t]}\ln{\frac{ \l^{u}_{ij}(s)}{g_{ij}}}\,dN_{ij}(s)-\int_{(0,t]}( \l^{u}_{ij}(s)-g_{ij})I_i(s)ds  \right\}},
\end{equation}
which is the solution of the following linear stochastic integral equation
\begin{equation}\label{L-sde-1}
L^u_t=1+\int_{(0,t]} L^u_{s^-}\underset{i,j:\, i\neq j}\sum I_i(s^-)\ell^u_{ij}(s)dM_{ij}(s),
\end{equation}
where $(M_{ij})_{ij}$ is the $P$-martingale given in \eqref{mart-1}. We have the following Girsanov-type result.  

\begin{proposition}[{\bf Girsanov density}]\label{L-u-mart} Assume $\l^u$ and $\xi$ satisfy (B1)-(B4). Then, $L^u$ is a $P$-martingale. Furthermore, $dP^u:=L^u(T) dP$.
\end{proposition}
\no For a  proof  of the proposition see the appendix.  

An important consequence of this proposition is that, under $P^u$, the  processes 
\begin{equation}\label{m-u-bsde}
M^u_{ij}(t):=M_{ij}(t)-\int_{(0,t]} \ell^u_{ij}(s)I_i(s^-)g_{ij}ds
\end{equation}
are zero mean, square integrable and mutually orthogonal martingales  whose  predictable quadratic variations are
\begin{equation}\label{m-u-bsde-q}
\langle  M^u_{ij}\rangle_t=\int_{(0,t]} I_i(s^-) \l^u_{ij}(s)ds.
\end{equation} 

In the next corollary we display an extension of Proposition \eqref{L-u-mart} to intensities involving the essential infimum and supremum of $\l^u$ w.r.t. $u$, that will be used below, such as the following case. Let $\phi=(\phi_{ij})_{ij}$ and $\bar\phi=(\bar\phi_{ij})_{ij}$ be predictable process and define the predictable process $\wh\l(t,x)=(\wh\l_{ij}(t,x))_{ij}$ (depending on $\phi,\bar\phi)$) given by  
\begin{equation}\label{l-hat}
\wh\l_{ij}(t,x):=\underset{u\in \U}{\mathrm{ess}\inf\,}\l_{ij}^u(t)\mathbf{1}_{\{\phi_{ij}(t)>\bar\phi_{ij}(t)\}}+\underset{u\in \U}{\mathrm{ess}\sup\,}\l_{ij}^u(t)\mathbf{1}_{\{\phi_{ij}(t)\le \bar\phi_{ij}(t)\}}, \,\,\, i\neq j.
\end{equation}
In view of \eqref{x-square-g}, it is readily seen that $\wh\l$ satisfies similar conditions as (B1)-(B3). In particular, 
\begin{equation*}
\underset{i\neq j}\sum\,\wh\l_{ij}(t,x)\le C(1+|x|_t+\underset{u\in \U}{\sup\,}\|P^u\|_2)\le C_T(1+\|\xi\|_2+|x|_t),
\end{equation*} 
where $C_T$ is a positive constant independent of $u$. Moreover, it is easily seen that the predictable process  $\wh\ell(t,x)=(\wh\ell_{ij}(t,x))_{ij}$ defined by 
$$
\wh{\ell}_{ij}(t,x)=\frac{\wh\l_{ij}(t,x)}{g_{ij}}-1,\,\,i\neq j,\,\quad \wh{\ell}_{ii}(t,x)=0,
$$
satisfies
\begin{equation}\label{l-hat}
\wh{\ell}_{ij}(t,x):=\underset{u\in \U}{\mathrm{ess}\inf\,}\ell_{ij}^u(t)\mathbf{1}_{\{\phi_{ij}(t)>\bar\phi_{ij}(t)\}}+\underset{u\in \U}{\mathrm{ess}\sup\,}\ell_{ij}^u(t)\mathbf{1}_{\{\phi_{ij}(t)\le \bar\phi_{ij}(t)\}}, \,\,\, i\neq j.
\end{equation}
Mimicking the proof of Proposition \eqref{L-u-mart} we obtain the following

\begin{corollary}\label{L-l-hat} Let $\wh\l$ be an intensity process satisfying
\begin{equation}\label{l-hat-est}
\underset{i\neq j}\sum\,\wh\l_{ij}(t,x)\le C_T(1+\|\xi\|_2+|x|_t),
\end{equation} 
where $C_T$ is a positive constant independent of $u$.
Then the  Dol{\`e}ans-Dade process $\wh L$ given by
\begin{equation*}
\wh{L}_t=1+\int_{(0,t]} \wh{L}_{s^-}\underset{i,j:\, i\neq j}\sum I_i(s^-)\wh\ell_{ij}(s)dM_{ij}(s)
\end{equation*}
is a true martingale. Moreover, under the probability measure $\wh{P}$ defined by $d\wh{P}:=\wh{L}(T) dP$, the  processes 
\begin{equation}\label{m-hat-bsde}
\wh{M}_{ij}(t):=M_{ij}(t)-\int_{(0,t]} \wh{\ell}_{ij}(s)I_i(s^-)g_{ij}ds
\end{equation}
are zero mean, square integrable and mutually orthogonal martingales  whose  predictable quadratic variations are
\begin{equation}\label{m-hat-bsde}
\langle  \wh{M}_{ij}\rangle_t=\int_{(0,t]} I_i(s^-)\wh\l_{ij}(s)ds.
\end{equation} 
\end{corollary} 

\subsection{An entropic BSDE characterization of the risk sensitive payoff} 
In the next proposition we show that the risk sensitive payoff functional $J(u)$ can be expressed in terms of the unique solution of an entropic Markov chain BSDE. 

In the rest of the paper, we will assume that $\F_0$ is the trivial $\s$-algebra which implies that $x(0)$ is a given deterministic point $x_0$ in $I$.
 
Given $z\in \R^{I\times I}$, the set of real-valued $I\times I$-matrices, we introduce the Hamiltonian associated with the optimal control problem (\ref{J-u})
\begin{equation}\label{ham-u}
H(t,x,u,z):=f(t,x,P^u\circ x^{-1}(t),u(t))+ \langle \ell^u(t),e^{z}-1\rangle_g,
\end{equation}
where
$$
\langle\ell^u(t),e^z-1\rangle_g:=\underset{i,j:\, i\neq j}\sum (e^{z_{ij}}-1)\ell^u_{ij}(t)g_{ij}\mathbf{1}_{\{x(t^-)=i\}}.
$$

\begin{proposition}\label{u-bsde} Assume the conditions (B1)-(B4) hold and that  $f$ and $h$ are uniformly bounded. Then, for any admissible control $u\in \U$,  the entropic BSDE 
\begin{equation}\label{u-yz-bsde}\left\{\begin{array}{lll}
-dY^u(t)=\{ H(t,x,u,Z^u(t))+\langle \tau(Z^u(t)),1\rangle_{g}\} dt-Z^u(t)dM(t),\\ \quad 
Y^u(T)=h(x(T),P^u\circ x^{-1}(T)).
\end{array}
\right.
\end{equation} 
admits a unique solution $(Y^u,Z^u)$ for which $Y^u$ is bounded and 
\begin{equation}\label{u-bsde-estim}
E\left[\int_{(0,T]} \|e^{Z^u(s)}-1\|_{g}^2ds\right]<\infty .
\end{equation}

\ms Furthermore, $Y^u$ is explicitly given by the formula 
\begin{equation}\label{Y-u}
e^{Y^u(t)}=E^u\left[\exp{\{h(x_T,P^u\circ x_T^{-1})+\int_t^T f(s,x,P^u\circ x_s^{-1},u_s)ds\}}\large|\F_t\right].
\end{equation}
In particular, $J(u)=e^{Y^u(0)}$.
\end{proposition}

\begin{proof}
Let $|f|_{\infty}$ and $|h|_{\infty}$ denote the uniform bounds of $f$ and $h$. Performing the change of variable $y^u(t):=e^{Y^u(t)+\int_0^t f(s,x,P^u\circ x^{-1}(s),u(s))ds}, \,\, \psi^u(t):=e^{Z^u(t)}-1$ and applying It\^{o}'s lemma, the process $(y^u,\psi^u)$ satisfies the BSDE
$$
\left\{\begin{array}{lll}
-dy^u(t)=\langle \ell^u(t),y^u(t^-)\psi^u(t)\rangle_g dt-y^u(t^-)\psi^u(t)dM(t),\\ 
\quad y^u(T)=\exp{\{h(x_T,P^u\circ x_T^{-1})+\int_0^Tf(s,x,P^u\circ x^{-1}(s),u(s))ds\}},
\end{array}
\right.
$$
Letting further $\varphi^u(t):=y^u(t^-)\psi^u(t)$, the process $(y^u,\varphi^u)$ satisfies the following linear BSDE
\begin{equation}\label{bsde-u-linear}
\left\{\begin{array}{lll}
-dy^u(t)=\langle \ell^u(t),\varphi^u(t)\rangle_gdt-\varphi^u(t)dM(t),\\ \quad y^u(T)=\exp{\{h(x_T,P^u\circ x_T^{-1})+\int_0^Tf(s,x,P^u\circ x^{-1}(s),u(s))ds\}},
\end{array}
\right.
\end{equation}
for which the driver $\langle \ell^u(t),z\rangle_g$ is (stochastic) Lipschitz:
$$
|\langle \ell^u(t),z^1\rangle_g-\langle \ell^u(t),z^2\rangle_g|\le \|\ell^u(t)\|_g \|z^1-z^2\|_g(t),
$$
where, in view of \eqref{lambda-norm}, $\|\ell^u(t)\|_g \in L^2([0,T]\times \Omega,dt\otimes dP)$. By Proposition \eqref{bsde-th}, the BSDE \eqref{bsde-u-linear} admits a unique solution for which 
$$
E\left[\int_{(0,T]} \|\varphi^u(s)\|^2_gds\right]<\infty.
$$
Furthermore, noting that by \eqref{m-u-bsde}, under $P^u$, the process
$$
\int_{(0,t]}\varphi^u(s)dM^u(s)=\int_{(0,t]}\varphi^u(s)dM(s)-\int_{(0,t]}\langle \ell^u(s),\varphi^u(s)\rangle_gds
$$
is a martingale, we may  take the conditional expectation w.r.t. $\F_t$  to obtain that
\begin{equation}\label{y-u-exp}
y^u(t)=E^u\left[\exp{\{h(x_T,P^u\circ x_T^{-1})+\int_0^T f(s,x,P^u\circ x_s^{-1},u_s)ds\}}\large|\F_t\right] 
\end{equation}
which obviously satisfies following estimate
\begin{equation}\label{bounds-u-bsde}
 \ln{|y^u|_T} \le |h|_{\infty}+|f|_{\infty}T,\,\, \,P\text{-}\as
\end{equation} 
Since the transformation $(Y^u,Z^u)\mapsto (y^u,\varphi^u)$ defined by
$$
y^u(t)=e^{Y^u(t)+\int_0^tf(s,x,P^u\circ x_s^{-1},u_s)ds}, \quad Z^u(t):=\ln(({y^u})^{-1}(t^-)\varphi^u(t)+1),\,\, 0\le t\le T,
$$ 
is one-to-one, a unique solution $(Y^u,Z^u)$ to the BSDE \eqref{u-yz-bsde} thus exists and, due to \eqref{y-u-exp}, $Y^u$ satisfies \eqref{Y-u}.  In particular, $J(u)=e^{Y^u(0)}$, since $\F_0$ is the trivial $\sigma$-algebra.

Since $|Y^u|_t\le \ln|y^u|_t$, using the estimate \eqref{bounds-u-bsde}, $Y^u$ is indeed bounded. Moreover, we have
$$
E\left[\int_{(0,T]} \|e^{Z^u(s)}-1\|_{g}^2ds\right]<\infty.
$$
\end{proof}

\subsection{Existence of an optimal control}
In the remaining part of this section we want to characterize controls $u^*\in\U$ such that $u^*=\arg\min_{u\in\U}J(u)$. A way to find such an optimal control is to proceed as in Proposition \ref{u-bsde} and introduce an entropic BSDE whose solution $Y^*$ satisfies $Y^*_0=\inf_{u\in\U}J(u)$. Then, by comparison, the problem can be reduced to  minimizing the corresponding Hamiltonian $H$ given by \eqref{ham-u} and the terminal value $h$ w.r.t. the control $u$.  

\ms \no Let $\mathbb{L}$ denote the $\sigma$-algebra of progressively measurable sets on $[0,T]\times\Omega$. For  $z\in \R^{I\times I}$, the set of real-valued $I\times I$-matrix, set
\begin{equation}\label{def-H}
H(t,x,z,u):=H(t,x,P^u\circ x_t^{-1},z,u_t).
\end{equation}
Since $H$ is continuous in $z$ and a progressively measurable process, it is an $\mathbb{L}\times B(\R^{I\times I})$-random variable. We have
\begin{proposition}\label{ess-inf}
There exists an $\mathbb{L}$-measurable process $H^*$ such that, for every $z\in \R^{I\times I}$,
\begin{equation}\label{u-opt-1}
H^*(t,x,z)=\mathrm{ess}\inf_{u\in \U}H(t,x,z,u),\quad dP \times dt \mbox{-a.s.} 
\end{equation} 
\end{proposition}
The proof of \eqref{u-opt-1} is similar to the one of Proposition 4.4 in \cite{DH}. We give it in an appendix for the sake of completeness. \\

\medskip Define the $\F_T$-measurable random variable
\begin{equation}\label{h-*}
h^*(x):=\underset{u\in\U}{\mathrm{ess}\inf\,} h(x(T),P^u\circ x^{-1}(T)).
\end{equation}

\begin{proposition} \label{bsde-*} Assume the conditions (B1)-(B4) hold and that  $f$ and $h$ are uniformly bounded. Then, there exists a unique solution $(Y^*,Z^*)$ of the entropic BSDE
\begin{equation}\label{opt-bsde}
\left\{\begin{array}{lll}
-dY^*(t)=\{H^*(t,x,Z^*(t))+\langle \tau(Z^*(t)),1\rangle_{g}\}\}dt-Z^*(t)dM(t),\\ Y^*(T)=h^*(x),
\end{array}
\right.
\end{equation}
such that  
\begin{equation}\label{bsde-*-estimate}
E\left[\int_{(0,T]} \|e^{Z^*(s)}-1\|^2_gds\right]<\infty,
\end{equation}
and there exists a probability measure $\wh P$ on $(\Omega,\F)$ which is absolutely continuous w.r.t. $P$, under which $Y^*$ admits the representation
\begin{equation}\label{Y-*-rep}
e^{Y^*(t)}=\wh{E}\left[\exp{\{h^*(x)+\int_t^T \underset{u\in\U}{\mathrm{ess}\inf\,} f(s,x,u)ds\}}|\F_t\right].
\end{equation}
\end{proposition}

\begin{proof}  Let $|f|_{\infty}$ and $|h|_{\infty}$ denote the uniform bounds of $f$ and $h$. Again, performing the change of variable $y(t):=e^{Y^*(t)}, \,\, \psi(t):=e^{Z^*(t)}-1$ and applying It\^{o}'s lemma, the process $(y,\psi)$ satisfies the BSDE
$$
y(t)=e^{h^*(x)}+\int_{(t,T]} y(s^-)H^*(s,x, \ln(\psi(s)+1))ds-\int_{(t,T]}y(s^-)\psi(s)dM(s).
$$
Since $y(t)$ is non-negative, we have
\begin{equation}\label{H*-F}
y(t^-)H^*(t,x, \ln(\psi(t)+1))=\mathrm{ess}\inf_{u\in \U}\left\{y(t^-)f(t,x,u)+\langle \ell^u(t),y(t^-)\psi(t)\rangle_g\right\}.
\end{equation}
Therefore, letting further $\varphi(t):=y(t^-)\psi(t)$, the process $(y,\varphi)$ satisfies the following linear BSDE

\begin{equation}\label{opt-bsde-linear}
\left\{\begin{array}{lll}
-dy(t)=F(t,x, y(t^-),\varphi(t))dt-\varphi(t)dM(t),\\ \quad y(T)=e^{h^*(x)},
\end{array}
\right.
\end{equation}
where
\begin{equation}\label{linear-driver}
F(t,x,y,z):=\mathrm{ess}\inf_{u\in \U}\left\{yf(t,x,u)+\langle \ell^u(t),z\rangle_g\right\}.
\end{equation}
We will now check that $F$ satisfies \eqref{driver-linear}. More precisely, we will show that there exists a predictable process $\wh\ell(t,x,z,\bar{z})$ satisfying (A2) such that
\begin{equation}\label{driver-F}
F(t,x,y,z)=\underset{u\in \U}{\mathrm{ess}\inf\,}f(t,x,u)y+\langle \wh{\ell}(t,x,z,0),z\rangle_g.
\end{equation}
We claim that, for every $z, \bar z\in \R^{I\times I}$, it holds that (cf. \cite{Cohen2015} pp. 482-483)
\begin{equation}\label{driver-F-1}
F(t,x,y,z)-F(t,x,y,\bar z)=\langle \wh{\ell}(t,x,z,\bar{z}),z-\bar{z}\rangle_g,
\end{equation}
where $\wh\ell(t,x,z,\bar{z})=(\wh{\ell}_{ij}(t,x,z,\bar{z}))_{ij}$ is given by  
\begin{equation}\label{ell-F}
\wh\ell(t,x,z,\bar{z})=\alpha(t,x,z,\bar{z}) \bar{\ell}(t,x,z,\bar{z})+(1-\alpha(t,x,z,\bar{z}))\underline{\ell}(t,x,z,\bar{z}),
\end{equation}
with 
\begin{equation}\label{alpha-F}
\alpha(t,x,z,\bar{z})=\frac{F(t,x,y,z)-F(t,x,y,\bar{z})-\langle \underline{\ell}(t,x,z,\bar{z}),z-\bar{z}\rangle_g}{ \langle \bar{\ell}(t,x,z,\bar{z})-\underline{\ell}(t,x,z,\bar{z}),z-\bar{z}\rangle_g} \in [0,1],
\end{equation}
where, $\underline{\ell}=(\underline{\ell}_{ij})_{ij},\,\, \overline{\ell}=(\overline{\ell}_{ij})_{ij})$ satisfy
\begin{equation}\label{ell-F-1}
\langle \underline{\ell}(t,x,z,\bar{z}),z-\bar{z}\rangle_g\le F(t,x,y,z)-F(t,x,y,\bar{z})\le \langle \overline{\ell}(t,x,z,\bar{z}),z-\bar{z}\rangle_g,
\end{equation}
with, for $i\neq j$,
$$
\begin{array}{lll}
\underline{\ell}_{ij}(t,x,z,\bar{z}):=\underset{u\in \U}{\mathrm{ess}\inf\,}\ell_{ij}^u(t)\mathbf{1}_{\{z_{ij}>\bar{z}_{ij}\}}+\underset{u\in \U}{\mathrm{ess}\sup\,}\ell_{ij}^u(t)\mathbf{1}_{\{z_{ij}\le \bar{z}_{ij}\}},\\ \overline{\ell}_{ij}(t,x,z,\bar{z}):=\underset{u\in \U}{\mathrm{ess}\inf\,}\ell_{ij}^u(t)\mathbf{1}_{\{\bar{z}_{ij}>z_{ij}\}}+\underset{u\in \U}{\mathrm{ess}\sup\,}\ell_{ij}^u(t)\mathbf{1}_{\{\bar{z}_{ij}\le z_{ij}\}}.
\end{array}
$$
Indeed, we have
$$
\begin{array}{lll}
F(t,x,y,z)-F(t,x,y,\bar{z})\ge \underset{u\in \U}{\mathrm{ess}\inf\,} \langle \ell^u(t),z-\bar{z}\rangle_g\\ \qquad\qquad \ge \langle \underset{u\in \U}{\mathrm{ess}\inf\,}\ell^u(t),(z-\bar{z})^+\rangle_g -
\langle \underset{u\in \U}{\mathrm{ess}\sup\,}\ell^u(t),(z-\bar{z})^-\rangle_g \\
\qquad\qquad = \langle \underline{\ell}(t,x,z,\bar{z}),z-\bar{z}\rangle_g,
\end{array}
$$
where, $(z-\bar{z})^{\pm}:=\left((z_{ij}-\bar{z}_{ij})^{\pm}\right)_{ij}$, with $\rho^+=\max(\rho,0)$ and $\rho^-=\max(-\rho,0),\,\rho\in \R$.  By symmetry, we also have
$$
F(t,x,y,z)-F(t,x,y,\bar{z})\le \langle \overline{\ell}(t,x,z,\bar{z}),z-\bar{z}\rangle_g.
$$
Combining these two inequalities and choosing $\wh\ell$ as in \eqref{ell-F}, we obtain \eqref{driver-F-1}. Moreover, since for every $u\in\U,\,\,i\neq j$, $\ell^u_{ij}=\frac{\l^u_{ij}}{g_ij}-1$,  the intensity processes defined by 
$$
\begin{array}{lll}
\underline{\lambda}_{ij}(t,x,z,\bar{z}):=\underset{u\in \U}{\mathrm{ess}\inf\,}\l_{ij}^u(t)\mathbf{1}_{\{z_{ij}>\bar{z}_{ij}\}}+\underset{u\in \U}{\mathrm{ess}\sup\,}\l_{ij}^u(t)\mathbf{1}_{\{z_{ij}\le \bar{z}_{ij}\}},\\
\overline{\lambda}_{ij}(t,x,z,\bar{z}):=\underset{u\in \U}{\mathrm{ess}\inf\,}\l_{ij}^u(t)\mathbf{1}_{\{\bar{z}_{ij}>z_{ij}\}}+\underset{u\in \U}{\mathrm{ess}\sup\,}\l_{ij}^u(t)\mathbf{1}_{\{\bar z_{ij}\le z_{ij}\}},
\end{array}
$$ 
are related to $\underline{\ell}_{ij}$ and $\overline{\ell}_{ij}$ by the formula
\begin{equation}\label{lambda-ell-bar}
\underline{\lambda}_{ij}(t,x,z,\bar{z})=(1+\underline{\ell}_{ij}(t,x,z,\bar{z}))g_{ij},\quad \overline{\lambda}_{ij}(t,x,z,\bar{z}):=(1+\overline{\ell}_{ij}(t,x,z,\bar{z}))g_{ij}.
\end{equation}
Thus, $\underline{\ell}_{ij}(t,x,z,\bar{z})$ and $\overline{\ell}_{ij}(t,x,z,\bar{z})$  are both strictly larger than $-1$.  From \eqref{lambda-ell-bar},  it follows that the intensity process $\wh\l(t,x,z,\bar{z})$ associated to $\wh\ell(t,x,z,\bar{z})$ through the formula  $\wh\l_{ij}(t,x,z,\bar{z}):=(1+\wh\ell_{ij}(t,x,z,\bar{z}))g_{ij},\,\,i\neq j,$ reads
$$
\wh\l_{ij}(t,x,z,\bar{z})=\alpha(t,x,z,\bar{z}) \bar{\l}_{ij}(t,x,z,\bar{z})+(1-\alpha(t,x,z,\bar{z}))\underline{\l}_{ij}(t,x,z,\bar{z}),
$$
and satisfies \eqref{l-hat-est} of Corollary \eqref{L-l-hat}. Now, since $F(t,x,y,0)=\underset{u\in \U}{\mathrm{ess}\inf\,}f(t,x,u)y$, we obtain \eqref{driver-F}. Therefore, by  Proposition \eqref{bsde-th}, the BSDE \eqref{opt-bsde-linear} admits a unique solution $(y,\varphi)$ which satisfies 
$$
E\left[\int_{(0,T]} \|\varphi(s)\|^2_g ds\right]<\infty,
$$
and 
\begin{equation}\label{y-*-rep}
y(t)=\wh{E}\left[\exp{\{h^*(x)+\int_t^T \underset{u\in\U}{\mathrm{ess}\inf\,} f(s,x,u)ds\}}|\F_t\right],
\end{equation}
where the expectation is taken w.r.t. $\wh{P}$ associated with $\wh\ell(t,x,\varphi,0)$, from which we obviously obtain the following estimate:
\begin{equation}\label{bounds-bsde}
\ln{|y|_T} \le |h|_{\infty}+|f|_{\infty}T,\quad P\text{-}\as
\end{equation} 
Since the transformation 
$(y,\varphi)\mapsto (Y^*,Z^*)$ defined by  
$$
y(t)=e^{Y^*(t)},\quad  Z^*(t):=\ln(y^{-1}(t^-)\varphi(t)+1),\quad 0\le t\le T,
$$ 
is one-to-one, a unique solution $(Y^*,Z^*)$ to the BSDE \eqref{opt-BSDE} thus exists. Moreover, using the representation \eqref{y-*-rep}, we obtain \eqref{Y-*-rep}.

\end{proof}

\ms We have the following comparison result.
\begin{proposition}[Comparison]\label{comp-bsde-u}
For every $t\in[0,T]$, it holds that
\begin{equation}
Y^*(t)\le Y^u(t),\quad P\mbox{-a.s.},\quad u\in\U.
\end{equation}
\end{proposition}

\begin{proof}  The result follows from Proposition \eqref{comparison} since $h^*(x)\le h(x_T,P^u\circ x_T^{-1})$
and the driver 
$F(t,x,y,z)=\mathrm{ess}\inf_{u\in \U}\left\{yf(t,x,u)+\langle \ell^u(t),z\rangle_g\right\}$ of the BSDE solved by $e^{Y^*}$  and the driver $b(t,x,y,z,u):=yf(t,x,u)+\langle \ell^u(t),z\rangle_g$ of the BSDE solved by $e^{Y^u}$ obviously satisfy  $F(t,x,y,z)\le b(t,x,y,z,u)$. Moreover, both drivers satisfy \eqref{balancing}.
\end{proof}

\begin{proposition}[$\ep$-optimality]\label{e-optimality} Assume that for any $\ep>0$ there exists $u^{\ep}\in \U$ such that $P$-a.s.,
\begin{equation}\label{e-*-H-g}
\left\{ \begin{array}{ll} H^*(t,x,Z^*(t))\ge H(t,x,Z^*(t),u^{\ep})-\ep, \quad 0\le t<T, \\ h^*(x)\ge h(x(T),P^{u^{\ep}}\circ x^{-1}(T))-\ep.
\end{array}
\right.
\end{equation}
Then, 
\begin{equation}\label{*-e-opt}
Y^*(t)=\underset{u\in\U}{\mathrm{ess}\inf\,} Y^u(t),\quad 0\le t\le T.
\end{equation}
\end{proposition}

\begin{proof} It suffices the show that 
$$
e^{Y^*(t)}=\underset{u\in\U}{\mathrm{ess}\inf\,} e^{Y^u(t)},\quad 0\le t\le T.
$$
Set $y^*(t)=e^{Y^*(t)}$ and $y^{u^{\ep}}=e^{Y^{u^{\ep}}}$. As above, $(y^*,\varphi^*)$  solves a BSDE with final value $\zeta^*:=e^{h^*(x)}$ and driver 
$$
F^*(t,x,y,z):=\underset{u\in\U}{\mathrm{ess}\inf\,}\{f(t,x,u)y+\langle \ell^u(t),z\rangle_g\},
$$ 
while $(y^{u^{\ep}},\varphi^{u^{\ep}})$ solves a linear BSDE with final value  $\z^{\ep}:=e^{h(x(T),P^{u^{\ep}}\circ x^{-1}(T))}$ and driver 
$$
b(t,x,y,z,u^{\ep}):=f(t,x,u^{\ep})y+\langle \ell^{u^{\ep}}(t),z\rangle_g.
$$ 
Moreover, by \eqref{H*-F}, the inequalities  \eqref{e-*-H-g} translate into 
\begin{equation}\label{e-*-exp}
\left\{ \begin{array}{ll}F^*(t,x,y^*(t),\varphi^*(t))\ge b(t,x,y^*(t),\varphi^*(t),u^{\ep})-\ep y^*(t), \quad 0\le t<T,\\ 
\zeta^*\ge e^{-\ep}\zeta^{\ep}.
\end{array}
\right.
\end{equation}
By adding and subtracting $ F(t,x,y^*(t),\varphi^*(t),u^{\ep})$ and $ F(t,x,y^*(t),\varphi^{u^{\ep}}(t),u^{\ep})$ respectively, and integrating by parts, using the linear structure of $F$, we obtain
$$
\begin{array}{lll}
y^*(t)-y^{u^{\ep}}(t)\ge  \int_{(t,T]}e^{\int_t^s f(r,x,u^{\ep})dr}\langle \ell^{u^{\ep}}(s), \varphi^*(s)-\varphi^{u^{\ep}}(s))\rangle_gds \\ \qquad\qquad\qquad -\int_{(t,T]}e^{\int_t^s f(r,x,u^{\ep})dr}(\varphi^*(s)-\varphi^{u^{\ep}}(s))dM(s)\\ \qquad\qquad\qquad  +\zeta^{\ep}(e^{-\ep}-1)e^{\int_t^T f(r,x,u^{\ep})dr}-\ep\int_{(t,T]}e^{\int_t^s f(r,x,u^{\ep})dr}y^*(s)ds. 
\end{array}
$$
Taking conditional expectation w.r.t. $P^{u^{\ep}}$ and arranging terms, noting that $e^{-\ep}-1\ge -\ep$, we obtain
$$
y^*(t)\ge y^{u^{\ep}}(t)-\ep E^{u^{\ep}}\left[e^{|h|_{\infty}+\int_t^T f(r,x,u^{\ep})dr}+ \int_{(t,T]}e^{\int_t^s f(r,x,u^{\ep})dr}y^*(s)ds|\F_t\right].
$$
Finally, since by \eqref{y-*-rep}, $y^*(t)\le e^{|h|_{\infty}+T|f|_{\infty}}$, we obtain 
$$
y^*(t)\ge y^{u^{\ep}}(t)-\ep e^{|h|_{\infty}+T|f|_{\infty}}(1+|h|_{\infty}+T|f|_{\infty}).
$$
This in turn implies that,  for every $0\le t\le T$,  $y^*(t)\ge \underset{u\in\U}{\mathrm{ess}\inf\,}y^u(t)\,\,\, P\text{-}\as$ 
\end{proof}

In next theorem, we characterize the set of optimal controls associated with  (\ref{opt-J}) under the dynamics $P^u$.
\begin{theorem}[Existence of optimal control]\label{opt-BSDE} If there exists $u^*\in\U$ such that 
\begin{equation}\label{*-H-opt}
H^*(t,x,Z^*(t))=H(t,x,P^{u^*}\circ x^{-1}(t),Z^*(t),u^*(t)),\quad 0\le t< T,
\end{equation}
and
\begin{equation}\label{*-g-opt}
h^*(x)=h(x(T),P^{u^*}\circ x^{-1}(T)).
\end{equation}
Then, \begin{equation}\label{Y-opt}
Y^*(t)=Y^{u^*}(t)=\underset{u\in\U}{\mathrm{ess}\inf\,} Y^u(t),\quad 0\le t\le T.
\end{equation}
In particular, $Y_0^*=\inf_{u\in\U} J(u)=J(u^*)$.
\end{theorem}

\begin{proof} In view of Proposition \eqref{bsde-*}, the conditions \eqref{*-H-opt} and \eqref{*-g-opt} imply that $Y^*=Y^{u^*}$.  Due to \eqref{*-e-opt}, we obtain \eqref{Y-opt}. 
\end{proof}

\section{The two-players zero-sum game problem}\label{zero-sum}
In this section we consider a two-players zero-sum game.  Let $\U$ (resp. $\V$) be the set of admissible $U$-valued (resp. $V$-valued) control strategies for the first (resp. second) player, where $(U,\delta_1)$ and $(V,\delta_2)$ are compact metric spaces.\\

For $(u,v),(\bar u,\bar v)\in U\times V$, we set
\begin{equation}\label{delta-u-v}
\delta((u,v),(\bar u,\bar v)):=\delta_1(u,\bar u)+\delta_2(v,\bar v).
\end{equation}
The distance $\delta$ defines a metric on the compact space $U\times V$. 

Let $P$  be the probability measure on $(\Omega, \mathcal F)$ such that $\F_0$ is trivial under which $x$ is a time-homogeneous Markov chain such that $P\circ x^{-1}(0)=\delta_{x_0}$, where $x_0$ is a given point in $I$,  with a deterministic $Q$-matrix $(g_{ij})_{ij}$  satisfying \eqref{G} and (B4) i.e. it satisfies  $\underset{i,j: \, j\neq i}\sum |i-j|^2 g_{ij} < \infty$. As above, $\mathcal{F}_0$ is the trivial $\s$-algebra.  

For $(u,v)\in\U\times \V$, let $P^{u,v}$ be the measure on $(\Om,\F)$ defined by
 \begin{equation}\label{P-u-v}
 dP^{u,v}:=L_T^{u,v}dP,
 \end{equation}
where
\begin{equation}\label{P-u-v-density}
 L^{u,v}(t):=\underset{\substack{i,j\\ i\neq j} }\prod \exp{\left\{ \int_{(0,t]}\ln{\frac{ \l^{u,v}_{ij}(s)}{g_{ij}}}dN_{ij}(s)-\int_{(0,t]}( \l^{u,v}_{ij}(s)-g_{ij})I_i(s)ds  \right\}}, 
\end{equation}

\begin{equation}\label{u-v-lambda}
\l^{u,v}_{ij}(t):=\l_{ij}(t,x,P^u\circ x^{-1}(t),u(t),v(t)),\,\,\, i,j\in I,\,\, 0\le t\le T,
\end{equation}
satisfying the following assumptions. 
\begin{itemize}
 \item[(C1)] For any $Q\in\mathcal{P}(\Omega)$, $(u,v)\in\U\times \V$ and  $i,j\in I$, the process $((\l_{ij}(t, x,Q\circ x^{-1}(t), u(t),v(t)))_t$ is predictable and satisfies, for every $t\in[0,T],\,w\in \Omega,\,\mu\in \mathcal{P}(I)$ and $i\neq j$,
$$
 \underset{(u,v)\in U\times V}{\inf\,}\l_{ij}(t, w,\mu, u,v)>0.
$$

 \item[(C2)] For $p=1,2$ and for every $t\in[0,T]$, $w\in \Om,\, u\in U, v\in V$ and  $\mu\in \mathcal{P}_2(I)$, 
 $$
 \underset{i,j: \, j\neq i}\sum |j-i|^p\l_{ij}(t,w,\mu,u,v)\le C(1+|w|^p_t+\int|y|^p\mu(dy)).
 $$
  \item[(C3)]  For $p=1,2$ and for every $t\in[0,T]$, $w, \tilde w\in \Om, (u,v), (\tilde u,\tilde v)\in U\times V$ and  $\mu, \nu \in\mathcal{P}(I)$,
  $$
  \begin{array}{lll}
  \underset{i,j: \, j\neq i}\sum |j-i|^p|\l_{ij}(t,w,\mu,u,v)-\l_{ij}(t,\tilde w,\nu,\tilde u,\tilde v)| \le C(|w-\tilde w|^p_t+d^p(\mu,\nu)\\ \qquad\qquad\qquad\qquad\qquad\qquad+\delta^p((u,v),(\tilde u,\tilde v)).
  \end{array}
  $$
   \end{itemize}

\ms\no As in Proposition \eqref{L-u-mart}, these assumptions guarantee that  $P^{u,v}$ is a probability measure on  $(\Omega,\F)$ under which the coordinate process $x$ is a chain with intensity matrix $\l^{u.v}$. Let $E^{u,v}$  denote the expectation w.r.t. $P^{u,v}$. Moreover, a similar estimate as  \eqref{x-square-g} and \eqref{W-uv-2}  hold.
\begin{equation}\label{x-square-g-u-v}
\underset{(u,v)\in \U\times \V}{\sup\,}\|P^{u,v}\|_2^2=\underset{(u,v)\in \U\times \V}{\sup\,}E^{u,v}[|x|_T^2]\le  Ce^{CT}\left(1+x_0^2\right).
\end{equation}
For every $u,v\in U$, 
\begin{equation}\label{W-uv-3}
D_T(P^{u,v},P^{\bar{u},\bar{v}})\le Ce^{CT}\delta((u,v),(\bar{u},\bar{v})).
\end{equation} 
Combining \eqref{G}, (C3), \eqref{x-square-g-u-v} and \eqref{x-square-g-0}, we obtain the following

\begin{lemma}  There exists a positive constant $C_T$ independent of the controls $(u,v)$  such that
\begin{equation}\label{lambda-norm-u-v}
E\left[\int_{(0,T]}\mathrm{ess}\sup_{v\in \V}\mathrm{ess}\sup_{u\in \U}\|\l^{u,v}(t)\|^2_g+\mathrm{ess}\sup_{u\in \U}\mathrm{ess}\sup_{v\in \V}\|\l^{u,v}(t)\|^2_gdt\right]\le C_T.
\end{equation}
\end{lemma}

\ms Let $f$  be a measurable function from $[0,T]\times\Om\times\mathcal{P}_2(I)\times U\times V$ into $\R$ and $h$ be a measurable function from $I\times\mathcal{P}_2(I)$ into $\R$ such that 

\begin{itemize}
\item[(C5)] $f$ and $h$ are uniformly  bounded.   
\end{itemize}
Setting 
$$
f(t,x,u,v):=f(t,x,P^{u,v}\circ x^{-1}(t),u(t),v(t))
$$
The performance functional $J(u,v),\,(u,v)\in\U\times\V$, associated with the controlled Markov chain is 
\begin{equation}\label{J-u-v}
J(u,v):=E^{u,v}\left[\exp{\left\{\int_0^T f(t,x,u,v)dt+ h(x(T),P^{u,v}\circ x^{-1}(T))\right\}}\right].
\end{equation}
  \\
The zero-sum game we consider is between two players, where the first player (with control $u$) wants to minimize the payoff (\ref{J-u-v}), while  the second player (with control $v$) wants to maximize it.  The zero-sum game boils down to showing existence of a saddle-point for the game i.e.  to show existence of a pair $(\widehat u, \widehat v)$ of strategies such that  
\begin{equation}\label{J-u-v-hat}
J(\hat u, v) \le J(\widehat u, \widehat v)\le J(u,\widehat v)
\end{equation}
for each $(u, v)\in\U\times\V$.\\
 The corresponding optimal dynamics is given by the probability measure $\widehat P$ on $(\Om,\F)$ defined by
\begin{equation}\label{opt-P}
d\widehat P=L_T^{\widehat u, \widehat v}dP
\end{equation}
under which the chain has intensity $\l^{\widehat u, \widehat v}$.

\ms\noindent For $(t,w,\mu,u)\in [0,T]\times\Om\times\mathcal{P}_2(I)\times U\times V$ and  matrices $z=(z_{ij})$ with real-valued entries, we introduce the Hamiltonian associated with the optimal control problem (\ref{J-u-v})
\begin{equation}\label{ham-u-v}
H(t,w,\mu,,u,v,z):=f(t,w,\mu,u,v)+ \langle \ell(t, w,\mu,u,v),e^z-1\rangle_g,
\end{equation}
where we recall that 
$$
\ell_{ij}(t, w,\mu,u,v)=\frac{\l_{ij}(t, w,\mu,u,v)}{g_{ij}}-1,\,\,i\neq j,\quad \ell_{ii}(t, w,\mu,u,v)=0. 
$$

Next, let $z\in R^{I\times I}$ and set 
\begin{itemize}
\item $\underline{H}(t,x,z):=\underset{v\in\V}{\mathrm{ess}\sup}\, \underset{u\in\U}{\mathrm{ess}\inf}\, H(t,x,u,v,z)$
\item $\overline{H}(t,x,z):=\underset{u\in\U}{\mathrm{ess}\inf}\, \underset{v\in\V}{\mathrm{ess}\sup}\, H(t,x,u,v,z),$
\item $\underline{h}(x):=\underset{v\in\V}{\mathrm{ess}\sup}\, \underset{u\in\U}{\mathrm{ess}\inf}\, h(x(T), P^{u,v}\circ x^{-1}(T))$,
\item $\overline{h}(x):=\underset{u\in\U}{\mathrm{ess}\inf}\, \underset{v\in\V}{\mathrm{ess}\sup}\, h(x(T), P^{u,v}\circ x^{-1}(T))$,
\item $(\underline{Y},\underline{Z})$ the solution of the entropic BSDE associated with $(\underline{H}, \underline{h})$ and  $(\overline{Y},\overline{Z})$ the solution of the entropic BSDE associated with $(\overline{H}, \overline{h})$.
\end{itemize}
Following a similar proof as the one leading to Proposition \eqref{bsde-*}, the driver of the BSDE associated with $(e^{\underline{Y}}, e^{\underline{Z}}-1)$ reads
\begin{equation}\label{down-F}
\underline{F}(t,x,y,z):=\underset{v\in\V}{\mathrm{ess}\sup}\, \underset{u\in\U}{\mathrm{ess}\inf}\{f(t,x,u,v)y+\langle \ell^{u,v}(t),z\rangle_g\},
\end{equation}
and the one associated with $(e^{\overline{Y}}, e^{\overline{Z}}-1)$ is
\begin{equation}\label{up-F}
\overline{F}(t,x,y,z):=\underset{u\in\U}{\mathrm{ess}\inf}\, \underset{v\in\V}{\mathrm{ess}\sup}\{f(t,x,u,v)y+\langle \ell^{u,v}(t),z\rangle_g\}.
\end{equation}

To $\underline{F}$ we associate  $\wh\b(t,x,z,\bar{z})=(\wh{\b}_{ij}(t,x,z,\bar{z}))_{ij}$  given by  
\begin{equation}\label{beta-down-F}
\wh\b(t,x,z,\bar{z})=\r_{\b}(t,x,z,\bar{z}) \bar{\b}(t,x,z,\bar{z})+(1-\r_{\b}(t,x,z,\bar{z}))\underline{\b}(t,x,z,\bar{z}),
\end{equation}
with 
\begin{equation}\label{rho-down-F}
\r_{\b}(t,x,z,\bar{z})=\frac{\underline{F}(t,x,y,z)-\underline{F}(t,x,y,\bar{z})-\langle \underline{\b}(t,x,z,\bar{z}),z-\bar{z}\rangle_g}{ \langle \bar{\b}(t,x,z,\bar{z})-\underline{\b}(t,x,z,\bar{z}),z-\bar{z}\rangle_g} \in [0,1],
\end{equation}
where $\overline{\b}=(\overline{\b}_{ij})_{ij}$ and  $\underline{\b}=(\underline{\b}_{ij})_{ij}$ read, for $i\neq j$,
$$
\begin{array}{lll}
\overline{\b}_{ij}(t,x,z,\bar{z}):=\underset{v\in\V}{\mathrm{ess}\sup\,}\underset{u\in \U}{\mathrm{ess}\sup\,}\ell_{ij}^{u,v}(t)\mathbf{1}_{\{z_{ij}>\bar{z}_{ij}\}}+\underset{v\in\V}{\mathrm{ess}\inf\,}\underset{u\in \U}{\mathrm{ess}\inf\,}\ell_{ij}^{u,v}(t)\mathbf{1}_{\{z_{ij}\le \bar{z}_{ij}\}},\\ 
\underline{\b}_{ij}(t,x,z,\bar{z}):=\underset{v\in\V}{\mathrm{ess}\sup\,}\underset{u\in \U}{\mathrm{ess}\sup\,}\ell_{ij}^{u,v}(t)\mathbf{1}_{\{\bar{z}_{ij}>z_{ij}\}}+\underset{v\in\V}{\mathrm{ess}\inf\,}\underset{u\in \U}{\mathrm{ess}\inf\,}\ell_{ij}^{u,v}(t)\mathbf{1}_{\{\bar{z}_{ij}\le z_{ij}\}}.
\end{array}
$$

To $\overline{F}$ we associate $\wh\t(t,x,z,\bar{z})=(\wh{\t}_{ij}(t,x,z,\bar{z}))_{ij}$ given by  
\begin{equation}\label{theta-up-F}
\wh\t(t,x,z,\bar{z})=\r_{\t}(t,x,z,\bar{z}) \bar{\t}(t,x,z,\bar{z})+(1-\r_{\t}(t,x,z,\bar{z}))\underline{\t}(t,x,z,\bar{z}),
\end{equation}
with 
\begin{equation}\label{rho-up-F}
\r_{\t}(t,x,z,\bar{z})=\frac{\overline{F}(t,x,y,z)-\overline{F}(t,x,y,\bar{z})-\langle \underline{\t}(t,x,z,\bar{z}),z-\bar{z}\rangle_g}{ \langle \bar{\t}(t,x,z,\bar{z})-\underline{\t}(t,x,z,\bar{z}),z-\bar{z}\rangle_g} \in [0,1],
\end{equation}
where, $\overline{\t}=(\overline{\t}_{ij})_{ij}$ and  $\underline{\t}=(\underline{\t}_{ij})_{ij}$ are given by  ($i\neq j$)
$$
\begin{array}{lll}
\overline{\t}_{ij}(t,x,z,\bar{z}):=\underset{u\in\U}{\mathrm{ess}\sup\,}\underset{v\in \V}{\mathrm{ess}\sup\,}\ell_{ij}^{u,v}(t)\mathbf{1}_{\{z_{ij}\ge\bar{z}_{ij}\}}+\underset{u\in\U}{\mathrm{ess}\inf\,}\underset{v\in \V}{\mathrm{ess}\inf\,}\ell_{ij}^{u,v}(t)\mathbf{1}_{\{z_{ij}< \bar{z}_{ij}\}},\\ 
\underline{\t}_{ij}(t,x,z,\bar{z}):=\underset{u\in\U}{\mathrm{ess}\sup\,}\underset{v\in \V}{\mathrm{ess}\sup\,}\ell_{ij}^{u,v}(t)\mathbf{1}_{\{\bar{z}_{ij}\ge z_{ij}\}}+\underset{u\in\U}{\mathrm{ess}\inf\,}\underset{v\in \V}{\mathrm{ess}\inf\,}\ell_{ij}^{u,v}(t)\mathbf{1}_{\{\bar{z}_{ij}<z_{ij}\}}.
\end{array}
$$
We omit the proof of the next lemma as it is similar to \eqref{driver-F-1}.
\begin{lemma} \label{balancing-game} $\underline{F}$ and $\overline{F}$ are balanced: For every $z,\bar z\in\R^{I\times I}$, 
\begin{equation}\begin{array}{lll}
\underline{F}(t,x,y,z)-\underline{F}(t,x,y,\bar z)=\langle \wh\b(t,x,z,\bar{z}),z-\bar{z}\rangle_g,\\
\overline{F}(t,x,y,z)-\overline{F}(t,x,y,\bar z)=\langle \wh\t(t,x,z,\bar{z}),z-\bar{z}\rangle_g,
\end{array}
\end{equation}
where $\wh\b$ and $\wh\t$ are given by \eqref{beta-down-F} and \eqref{theta-up-F}.
\end{lemma}
Again, by a similar proof as the one leading to Proposition \eqref{bsde-*}, there exists a unique solution $(\underline{Y},\underline{Z})$ (resp.  $(\overline{Y},\overline{Z})$) to the entropic BSDE associated with $(\underline{H}, \underline{h})$ (resp. $(\overline{H}, \overline{h})$).

\begin{definition}[Isaacs' condition]
We say that the Isaacs' condition holds for the game if
$$
\left\{\begin{array}{lll}
\underline{H}(t,x,z)=\overline{H}(t,x,z),\quad 0\le t< T, \\
\underline{h}(x)=\overline{h}(x).
\end{array}
\right.
$$
\end{definition}

Due to Lemma \eqref{balancing-game} , we may apply the comparison theorem (Proposition \eqref{comparison}) to the BSDEs satisfied by $e^{\underline{Y}}$ and $e^{\overline{Y}}$, to obtain the following

\begin{proposition}\label{game-comparison} For every $t\in[0,T]$, it holds that $\underline{Y}_t\le \overline{Y}_t$, $\,P$-a.s.. Moreover, if the Issac's condition holds, then  
\begin{equation}\label{nash}
\underline{Y}(t)=\overline{Y}(t):=Y(t),\quad P\text{-a.s.},\quad 0\le t\le T.
\end{equation}
\end{proposition}  

In the next theorem, we formulate conditions for which the zero-sum game has a value.  

For $(u,v)\in\U\times \V$, let $(Y^{u,v},Z^{u,v})$ be the solution of the entropic BSDE associated to $(H,h)$: 
\begin{equation*}\label{u-v-yz-bsde}\left\{\begin{array}{lll}
-Y^{u,v}(t)=\{H(t,x,u,v,Z^{u,v}(t))+\langle \tau(Z^{u,v}(t)),1\rangle_g \}dt-Z^{u,v}(t)dM(t),\\
Y^{u,v}(T)=h(x(T),P^{u,v}\circ x^{-1}(T)),
\end{array}
\right.
\end{equation*} 

\begin{theorem}[Existence of a value of the game]\label{value-game}
Assume that, for every $0\le t<T$, 
$$
\underline{H}(t,x,Z(t))=\overline{H}(t,x,Z(t)).
$$
If there exists $(\widehat{u},\widehat{v})\in\U\times\V$ such that, for every $0\le t<T$, 
\begin{equation}\label{sp-H}
\underline{H}(t,x,Z(t))=\underset{u\in \U}{\mathrm{ess}\inf}\, H(t,x,u,\widehat{v},Z(t))=\underset{v\in \V}{\mathrm{ess}\sup}\, H(t,x,\widehat{u},v,Z(t)),
\end{equation}
and
\begin{equation}\label{sp-g}
\underline{h}(x)=\overline{h}(x)=\underset{u\in \U}{\mathrm{ess}\inf}\, h(x(T),P^{u,\widehat{v}}\circ x^{-1}(T))=\underset{v\in \V}{\mathrm{ess}\sup}\, \, h(x(T),P^{\widehat{u},v}\circ x^{-1}(T)).
\end{equation}
Then, 
\begin{equation}\label{value}
Y(t)=\underset{u\in \U}{\mathrm{ess}\inf}\,\underset{v\in \V}{\mathrm{ess}\sup}Y^{u,v}(t)=\underset{v\in \V}{\mathrm{ess}\sup}\,\underset{u\in \U}{\mathrm{ess}\inf}\,Y^{u,v}(t), \quad 0\le t\le T.
\end{equation}
Moreover, the pair $(\widehat{u},\widehat{v})$ is a saddle-point for the game.
\end{theorem}

\begin{proof} Let $(u,v)\in\U\times \V$ and $(\widehat Y^u,\widehat Z^u)$ and $(\widetilde Y^v,\widetilde Z^v)$ be the solution of the following entropic BSDE 
\begin{equation*}\label{u-v-yz-bsde}\left\{\begin{array}{ll}
-\widehat Y^{u}(t)=\{\underset{v\in \V}{\mathrm{ess}\sup}\,H(t,x,u,v,\widehat Z^{u}(t)) +\langle \tau(\widehat Z^{u}(t)),1\rangle_g\}dt-\widehat Z^{u}(t)dM(t),\\
\widehat Y^{u}(T)=\underset{v\in \V}{\mathrm{ess}\sup}\,h(x(T),P^{u,v}\circ x^{-1}(T)),
\end{array}
\right.
\end{equation*} 
\begin{equation*}\label{u-v-yz-bsde}\left\{\begin{array}{ll}
-\widetilde Y^{v}(t)=\{\underset{u\in \U}{\mathrm{ess}\inf}\, H(t,x,u,v,\widetilde Z^{v}(t))+\langle \tau(\widetilde Z^{v}(t)),1\rangle_g \} dt-\widetilde Z^{v}(t)dM(t),\\
\widetilde Y^{v}(T)=\underset{u\in \U}{\mathrm{ess}\inf}\,h(x(T),P^{u,v}\circ x^{-1}(T)).
\end{array}
\right.
\end{equation*} 
By uniqueness of the solutions of the BSDEs, we have
\begin{equation}\label{*-*}
\widehat Y^{u^*}(t)=\underset{v\in \V}{\mathrm{ess}\sup}\, Y^{u^*,v}(t),\quad \widetilde Y^{v^*}(t)=\underset{u\in \U}{\mathrm{ess}\inf}\, Y^{u,v^*}(t),
\end{equation}
and, by comparison, we have
$$
 \widehat Y^u(t)\ge Y(t) \ge \underset{v\in \V}{\mathrm{ess}\sup}\, \widetilde Y^v(t).
$$
Therefore,
$$
\underset{u\in \U}{\mathrm{ess}\inf}\, \widehat Y^u(t)\ge Y(t)\ge \underset{v\in \V}{\mathrm{ess}\sup}\, \widehat Y^v(t).
$$
But, by \eqref{sp-H}, \eqref{sp-g} and  the uniqueness of the solutions of these entropic  BSDEs, we have
$\widehat Y^{\widehat{u}}=Y=\widetilde Y^{\widehat{v}}$. Therefore, 
$$
\widehat{Y}^{\widehat{u}}(t)=\underset{u\in \U}{\mathrm{ess}\inf}\, \widehat Y^u(t)=Y(t)=\widetilde Y^{\widehat{v}}(t)=\underset{v\in \V}{\mathrm{ess}\sup}\,\widetilde Y^v(t)=Y^{\widehat{u},\widehat{v}}(t).
$$
Using \eqref{*-*}, we obtain
 $$
Y(t)=Y^{\widehat{u},\widehat{v}}(t)=\underset{v\in\V}{\mathrm{ess}\sup}\,Y^{\widehat{u}, v}(t)=\underset{u\in\U}{\mathrm{ess}\inf}\,Y^{u, \widehat{v}}(t).
$$
Therefore,
$$
Y^{u, \widehat{v}}(t)\le Y^{\widehat{u}, \widehat{v}}(t)\le Y^{\widehat{u}, v}(t).
$$
Thus, $Y^{\widehat{u},\widehat{v}}(t)$ is the value of the game and $(\widehat{u},\widehat{v})$ is a saddle-point.
\end{proof}

\subsection*{Concluding remarks}
\begin{enumerate}\label{selection}
\item In the control problem, if the marginal law $P^u\circ x_s^{-1}$ of $x_s$ under $P^u$ is a function of $u(s)$ only and does not depend on the whole path of $u$ over $[0,s]$, it suffices to take the minimum (provided continuity assumptions w.r.t. the control such as \eqref{W-uv-2} and \eqref{W-uv-3}) of  $H$ and $h$  over the compact set of controls $U$, instead of taking the essential infimum over $\U$. By the measurable selection theorem (see e.g. \cite{Benes}), an optimal control over $[0,T]$ can be obtained by pasting the minima of $H$ and $h$.  The same remark holds the the zero-sum game problem. 

\item It is possible to characterize the optimal controls $\hat u$ and the equilibrium points $(\hat u,\hat v)$ in terms of a stochastic maximum principle. This approach will be discussed in a future work.

\item The uniform boundedness assumptions imposed on  the functions $f$ and $h$
can be substantially weakened by using subtle arguments on existence and uniqueness of
solutions of  BSDEs which are by now well known in the BSDEs literature.

\end{enumerate}


\section{Appendix }
\begin{proof}[{\bf Proof of Proposition \eqref{L-u-mart}}]  The proof of the uniform integrability of $L^u$  is inspired by the proof of Proposition (A.1) in \cite{EH}.  As mentioned above, it suffices to prove that $E[L^u_T]=1$. For $n\ge 0$, let $\l^n$ be the predictable intensity matrix given by $\l^n_{ij}(t):=\l^u_{ij}(t)\mathbf{1}_{\{\omega,\,\,|x(\omega)|_t\le n\}}$ and let $L^n$ be the associated Dol{\`e}ans-Dade exponential and $P^n$ the positive measure defined by $dP^n=L^n_T dP$.  Noting that, for $i, j\in I, \,i\neq j$, $|i-j|\ge 1$, by (B3) and \eqref{x-square-g}, we have
$$
\underset{i\neq j}\sum\,\l^u_{ij}(t,x,P^u\circ x^{-1}_t,u(t))\le C_T(1+|x|_t+\|\xi\|_2).
$$
Thus, for every $n\ge 1$, $\l^n_{ij}(t)\le C_T(1+n+\|\xi\|_2)$, i.e. $\l^n_{ij}$ is bounded. In view of \cite{bremaud}, Theorem T11, $L^n$ is a $P$-martingale. In particular, $E[L^n_T]=1$ and $P^n$ is a probability measure. By \eqref{x-square-g}, $|x|_T<\infty,\, P$-a.s. Therefore, on the set $\{\omega,\,\,|x(\omega)|_T\le n_0\}$, for all $n\ge n_0$, $L^n_T(\omega)=L_T^u(\omega)$. This in turn yields that 
$L^n_T\to L_T^u,\,\,P$-a.s., as $n\to +\infty$. Now, if $(L^n_T)_{n\ge 1}$ is uniformly integrable, the $P$-a.s. convergence implies $L^1(P)$-convergence of $L^n_T$ to $L_T^u$, yielding  $E[L^u_T]=1$. It remains to show that $(L^n_T)_{n\ge 1}$ is uniformly integrable: 
$$
\lim_{a\to \infty}\sup_{n\ge 1}\int_{\{L^n_T>a\}} L^n_T\,dP=0.
$$
For $m\ge 1$, set $\t_m=\inf\{t\le T,\,\, |x|_t\ge m\}$ if the set is nonempty and $\t_m=T+1$ if it is empty. Denoting by $E^n$ the expectation w.r.t. $P^n$, we have
\begin{equation}\label{tau-m}\begin{array}{lll}
\int_{\{\t_m\le T\}} L^n_T\,dP=P^n(\t_m\le T)=P^n(|x|_T\ge m)\\ \qquad\qquad\qquad\quad \le E^n[|x|_T]/m\le C/m,
\end{array}
\end{equation}
where, by \eqref{x-square-g}, $C$ does not depend on $n$.  

\ms\no Let $\e>0$. Choose $m_0\ge 1$ such that $C/m_0<\e$. We have, for all $n\ge m_0$, $L^n_{T\wedge \t_{m_0}}=L^{m_0}_{T\wedge \t_{m_0}}$. This entails that
\begin{equation*}
\sup_{n\ge 1}\int_{\{L^n_{T\wedge \t_{m_0}}>a\}} L^n_{T\wedge \t_{m_0}}\,dP= \max_{n\le m_0}\int_{\{L^n_{T\wedge \t_{m_0}}>a\}} L^n_{T\wedge \t_{m_0}}\,dP \to 0,\,\, a\to  \infty.
\end{equation*}
So there exists $a_0>0$ such that whenever $a>a_0$, 
\begin{equation}\label{L-m}
\max_{n\le m_0}\int_{\{L^n_{T\wedge \t_{m_0}}>a\}} L^n_{T\wedge \t_{m_0}}\,dP< \e.
\end{equation}
We have
$$
\begin{array}{lll}
\underset{n\ge 1}\sup \int_{\{L^n_T>a\}} L^n_T\,dP\le\underset{n\ge 1}\sup\int_{\{L^n_T>a,\, \t_{m_0}\le T\}} L^n_T\,dP+\underset{n\ge 1}\sup\int_{\{L^n_T>a,\, \t_{m_0}> T\}} L^n_T\,dP \\ \qquad\qquad\qquad\qquad \le \underset{n\ge 1}\sup\int_{\{\t_{m_0}\le T\}} L^n_T\,dP+\underset{n\ge 1}\sup\int_{\{L^n_{T\wedge \t_{m_0}}>a,\, \t_{m_0}> T\}} L^n_{T\wedge \t_{m_0}}\,dP \\ \qquad\qquad\qquad\qquad \le \underset{n\ge 1}\sup\int_{\{\t_{m_0}\le T\}} L^n_T\,dP+\underset{n\le m_0 }\max\int_{\{L^n_{T\wedge \t_{m_0}}>a\}} L^n_{T\wedge \t_{m_0}}\,dP\\
\qquad\qquad\qquad\qquad\le C/m_0+\e < 2\e,
\end{array}
$$
 in view of \eqref{tau-m} and \eqref{L-m}. This finishes the proof since $\e$ is arbitrary.
 \end{proof}

\begin{proof}[{\bf Proof of Proposition \eqref{ess-inf}}]

For $n\ge 0$ let $z_n\in \mathbb{Q}^{I\times I}$, the $I\times I$-matrix with rational entries. Then, since $(t,\omega)\mapsto H(t,\omega,z_n,u)$ is $\mathbb{L}$-measurable, its essential infimum w.r.t. $u\in\U$ is well defined i.e. there exists a $\mathbb{L}$-measurable r.v. $H^n$ such that
\begin{equation}\label{H-n}
H^n(t,x, z_n)=\underset{u\in \U}{\mathrm{ess}\inf\,} H(t,x,z_n,u),\quad dP \times dt\mbox{-a.s.}
\end{equation}
Moreover, there exists a set $\mathcal{J}_n$ of $\U$ such that $(t,\omega)\mapsto \underset{u\in \mathcal{J}_n}{\inf\,} H(t,\omega,z_n,u)$ is $\mathbb{L}$-measurable and 
$$
H^n(t,x,z_n)=\underset{u\in \mathcal{J}_n}{\inf\,} H(t,x,z_n,u), \quad dP \times dt\mbox{-a.s.}
$$
Next,  set $N=\bigcup_{n\ge 0} N_n$, where
$$
N_n:=\{(t,\omega):\,\,H^n(t,\omega)\neq\underset{u\in \mathcal{J}_n}{\inf\,} H(t,\omega,z_n,u)\}.
$$
Then, $dP\otimes dt(N)=0$. \\ 

\noindent We define $H^*$ as follows: For $(t,\omega)\in N^c$  (the complement of $N$),
\begin{equation}\label{ess-inf-J}
H^*(t,x,z)=\left\{\begin{array}{ll}
\underset{u\in \mathcal{J}_n}{\inf\,} H(t,x,z_n,u) & \text{if }\,\, z=z_n\in\mathbb{Q}^{I\times I}, \\ \underset{z_n\to z}{\lim\,\,} \inf_{u\in \mathcal{J}_n} H(t,x,z_n,u) & \text{otherwise }.
\end{array}
\right.
\end{equation}
The last limit exists due to the fact that, for $n\neq m$, we have
$$
\begin{array}{ll}
|\underset{u\in \mathcal{J}_n}{\inf\,} H(t,x,z_n,u) -\underset{u\in \mathcal{J}_m}{\inf\,} H(t,x,z_m,u)|
=|H^*(t,x,z_n)-H^*(t,x,z_m)| \\ \quad \le \underset{u\in \U}{\mathrm{ess}\sup\,}\left|H(t,x,P^u\circ x^{-1}_t,z_n,u)-H(t,x,P^u\circ x^{-1}_t,z_m^,u)\right |\\ \qquad\qquad\le C(1+|x|_t+ \sup_{u\in U}\|P^u\|_2)\|e^{z_n}-e^{z_m}\|_g(t).
\end{array}
$$

\noindent  We now  show that, for every $z\in\R^{I\times I}$,
\begin{equation}\label{H-*}
H^*(t,x,z)=\underset{u\in \U}{\mathrm{ess}\inf\,} H(t,x,z,u),\quad dP \times dt\mbox{-a.s.}
\end{equation}
If $z\in\mathbb{Q}^{I\times I}$, the equality follows from the definitions \eqref{H-n} and \eqref{ess-inf-J}. Assume $z\notin\mathbb{Q}^{I\times I}$ and let $z_n\in\mathbb{Q}^{I\times I}$ such that $z_n\to z$. Further, let $\varphi(t,x_.)$ be a progressively measurable process such that $\varphi(t,x_.)\le H(t,x_.,z,u)$ for all $u\in\U$. Thus, for every $\e>0$ there exists $n_0\ge 0$ such 
$$
\begin{array}{ll}
\varphi(t,x)\le H(t,x,z_n,u)+\e,\quad n\ge n_0, \,\, u\in\U.
\end{array}
$$
Therefore, $\varphi(t,x_.)\le H^*(t,x_.,z_n)+\e,\,\,n\ge n_0$. Letting $n\to \infty$, we obtain
$\varphi(t,x_.)\le H^*(t,x_.,z)+\e$. Sending $\e$ to $0$, we finally get $\varphi(t,x_.)\le H^*(t,x_.,z)$, i.e.
$$
\underset{u\in \U}{\mathrm{ess}\inf\,} H(t,x,z,u)\le H^*(t,x,z),\quad dP \times dt\mbox{-a.s.}
$$
On the other hand, in view of \eqref{ess-inf-J} and the linearity of $H$ in $z$, we have $H^*(t,x,z)\le H(t,x,z,u),\,\, u\in\U$. Thus,
\begin{equation*}
H^*(t,x,z)\le \underset{u\in \U}{\mathrm{ess}\inf\,}H(t,x,z,u).
\end{equation*}
This finishes the proof of \eqref{H-*}.
\end{proof}

\end{document}